\documentclass[12pt]{article}
\usepackage{amsmath}
\usepackage{amssymb}
\usepackage{latexsym}
\usepackage{enumerate}
\usepackage[ansinew]{inputenc}
\usepackage[all]{xy}
\usepackage{graphicx}
\usepackage[usenames]{color}

\newtheorem{theorem}{Theorem}[section]

\newtheorem{remark}[theorem]{Remark}

\newtheorem{definition}[theorem]{Definition}
\newenvironment{proof}{\noindent{\sc Proof.}}{\quad\qed\medskip}

\newcommand{\Z}{{\mathbb Z}}

\newcommand{\qed}{\quad\lower0.05cm\hbox{$\Box$}}

\newcommand{\downarrowright}[1]{\downarrow
\rlap{\raise0.1cm\hbox{$\scriptstyle{#1}$}}}

\newcommand{\downarrowleft}[1]{\rlap{\kern-0.2cm
\raise0.1cm\hbox{$\scriptstyle{#1}$}}\downarrow}

\newcommand{\uparrowright}[1]{\uparrow
\rlap{\lower0.1cm\hbox{$\scriptstyle{#1}$}}}

\newcommand{\uparrowleft}[1]{\rlap{\kern-0.2cm
\lower0.1cm\hbox{$\scriptstyle{#1}$}}\uparrow}

\newcommand{\ra}{\rightarrow}

\begin{document}
\setcounter{page}{1}
\title{Bredon homology of wallpaper groups}
\footnotetext{


The author was supported by grant PID2020-117971GB-C21 of the Spanish Ministery of Science and Innovation, and grant FQM-213 of the Junta de Andaluc\'{\i}a.}

\author{Ram\'on Flores}

\date{November 28, 2021}
\maketitle

\begin{abstract}

In this paper we compute the Bredon homology of wallpaper groups with respect to the family of finite groups and with coefficients in the complex representation ring. We provide explicit bases of the homology groups in terms of irreducible characters of the representation rings of the stabilizers.

\end{abstract}


\setcounter{equation}{0}

\section{Introduction}

Bredon homology is one of the main instances of equivariant homology theory. Roughly speaking (see definitions in Section \ref{Prelim}), given a $G$-space $X$ and a family $\mathcal{F}$ of stabilizers of the action, the homology groups depend on a coefficient module $N$ which takes values in abelian groups and takes account of the structure of the family $\mathcal{F}$. After their development by G. Bredon in the sixties for the case of $G$ finite, different choices of $\mathcal{F}$ have showed different roles of this homology theory in contexts as equivariant obstruction theory \cite{May96}, partition complexes \cite{ADL16}, stable homotopy \cite{GM95} or group dimension theory \cite{KMN11}. The computation of the Bredon homology groups has been particularly important in relation with the Isomorphism Conjectures (see a survey in \cite{LR05}), as they may permit the computation of $K$-theory groups via an equivariant version of the Atiyah-Hirzebruch spec
tral sequence.
In this framework, when $X$ is the classifying space $\underline{E}G$ for proper actions of $G$ (see Section \ref{Classifying} below),the Bredon homology of $\underline{E}G$ is an invariant of the group $G$, and in this case we will directly say the ``Bredon homology of $G$".

We deal in the sequel with the Bredon homology of the crystallographic groups of the plane (also called \emph{wallpaper groups}), with respect to the family of finite subgroups and with coefficients in the complex representation ring. This coefficient module codifies the complex representation theory of the finite subgroups of the group, and is relevant in relation with the Baum-Connes conjecture, the best known of the Isomorphism Conjectures. Recall that this statement identifies the equivariant $K$-homology of the space $\underline{E}G$ with the $K$-theory of the reduced $C^*$-algebra of $G$. The conjecture is true for wallpaper groups, as they are solvable (see Section \ref{Wallpaper}), and the corresponding values of $K_*(C^*_r G)$ were computed in his thesis by Yang (\cite{Yan97} (see also L\"{u}ck-Stamm \cite{LuSt00}, where in particular a little mistake in Yang's results is corrected). Other computations of Bredon homology in the context of Baum-Connes conjecture may be found for example in \cite{San08}, \cite{LORS18} or \cite{AnFo14}.

The main goal of this paper is to offer \emph{explicit} computations of the Bredon homology group of wallpaper groups. By ``explicit" we mean giving bases of the homology groups (as abelian groups) in terms of irreducible characters of representations of finite stabilizers of the action of the groups, as well as a detailed description of the Bredon complex and the corresponding differentials. Aside the information which is obtained in this way about the representation theory of the group, the motivation of the study has come from the following problem. Consider a group $G$ which is a colimit of wallpaper groups, and try to compute the left-hand side of Baum-Connes for $G$ (this is, for example, the case of different extensions of $SL(2,\Z)$ by $\Z^2$). A possible strategy is to obtain the Bredon homology of $G$ out of the Bredon homologies of the wallpaper groups involved in the colimit, but this computation involve a precise knowledge of the induced homomorphisms in homology, and in particular of concrete generators of each group in the diagram, which is the kind of information that our study provides. Moreover, we also expect that our results may permit a sharp description of the Baum-Connes' assembly map for wallpaper groups, as in for example \cite{ABGRW} \cite{FPV17} or \cite{Poo19}. Finally, we are aware that Yang and L\"{u}ck-Stamm results', and also the computational approach by Bui-Ellis \cite{BuEl16}, provide the isomorphism type (as abstract groups) of some of this Bredon homology groups, but as said above, we believe that the main contribution of the present paper is the explicit description of the groups.

The structure of the paper is as follows. In Section \ref{Prelim} we give the necessary information about wallpaper groups, classifying spaces, Bredon homology, representation theory and Smith normal forms in order to make the paper as self-contained as possible. In particular, Section \ref{Representation} contains all the relevant information about the representation theory of the stabilizers. Section \ref{Computations} contains all the computations of Bredon homology groups, with a little introduction in which we explain the steps we follow in each calculation.

\textbf{Acknowledgments}. We warmly thank D. Schattschneider for the permission to reproduce the picture (Figure \ref{Doris}) of the patterns of the wallpaper groups that appear in the nice paper \cite{Sch78}. We also thank J. Gonz\'alez-Meneses for helpful comments, and S. Pooya, A. Zumbrunnen and in particular A. Valette for providing the motivation for this work.

\section{Preliminaries}
\label{Prelim}

\subsection{Wallpaper groups}
\label{Wallpaper}

In this subsection we recall the main features of the wallpaper groups, which are the main object of study of this paper. Nice surveys of the theory can be found in \cite{Sch78} or \cite{Mac85}, while presentations by generators and relations for all the wallpaper groups are given in \cite{DdSS99}.

 We start with the definition:

\begin{definition}

A discrete $G$ of isometries of the plane $\mathbb{R}^2$ is called a \emph{wallpaper group} if the action of $G$ on the plane is properly discontinuous and the quotient $\mathbb{R}^2/G$ is compact.

\end{definition}

There are exactly seventeen non-isomorphic wallpaper groups, as was independently proved by Fedorov and Schoenflies. Every such group $G$ is in particular defined by an extension:

$$\Z^2\rightarrow G\ra  F,$$ where $F$ is a finite group, called the \emph{point group} of the wallpaper group. The generators of the free abelian group correspond to two independent translations, and the images of a certain compact pattern of the plane by these translations tessellate it (this is the reason of the name \emph{wallpaper}). In Figure \ref{Doris} such patterns are pictured for all these groups, and the fact that they tessellate, and then contain a fundamental domain for the action, will be often used implicitly in Section 4, when a representative of class of equivariant 2-cells for a $G$-CW complex structure is defined for every wallpaper group.

These groups can also possess rotations, reflections and glide-reflections, corresponding in particular rotations and reflections with torsion elements of the groups. In Figure \ref{Doris} the rotation centers contained in the pattern can be observed, as well as the reflection and glide-reflection axes. We will use this picture as a major source of information in the computations of Section \ref{Computations}.

In Table \ref{Wallpaper groups} we have compiled some relevant information about the wallpaper groups that will useful for us. In the second column there and throughout the paper we will denote by $D_n$ the dihedral group of $2n$ elements. In turn, the cyclic group of $n$ elements will be denoted indistinctly  by $C_n$ or $\Z /n$ in all the sequel. The third column of the table explain if the extension that defines the group splits or not. In the fourth the torsion primes of every group can be found: observe that $\mathbf{p1}$ and $\mathbf{pg}$ are the only torsion-free instances. Finally, in the fifth column, we distinguish 2- and 4- rotations when there are 2-rotations which in the groups that have no roots; and same for 2-,3- and 6-rotations in $\mathbf{p6}$ and $\mathbf{p6m}$.

In next section we will define the classifying space for proper actions, which is fundamental in our computations, and in particular we will recall the appropriate model of this space for wallpaper groups.

\begin{figure}
\label{Doris}
\includegraphics{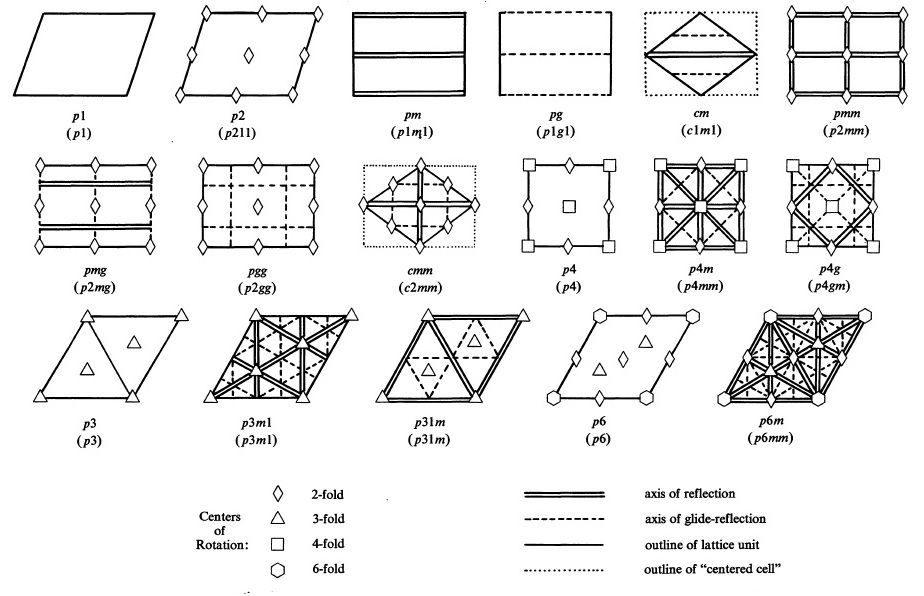}
\caption{Patterns for the wallpaper groups}
\end{figure}

\begin{table}[t]
\begin{center}
\begin{tabular}{| c | c | c | c | c | c | c |}
\hline
Group & Point group & Split & Torsion & Rotations & Reflections & Glide-reflections \\ \hline

$\mathbf{p1} $  & $ \{1\} $ & - & No & No & No & No \\ \hline
$\mathbf{p2} $ & $ C_2 $ & Yes & 2 & 2- & No & No \\ \hline
$\mathbf{pm} $ & $ C_2 $ & Yes & 2 & No & Yes & No \\ \hline
$\mathbf{pg} $ & $ C_2 $ & No & No & No & No & Yes \\ \hline
$\mathbf{cm} $ & $ C_2 $ & Yes & 2 & No & Yes & Yes \\ \hline
$\mathbf{pmm} $ & $ D_2 $ & Yes & 2 & 2- & Yes & No \\ \hline
$\mathbf{pmg} $ & $ D_2 $ & No & 2 & 2- & Yes & Yes \\ \hline
$\mathbf{pgg} $ & $ D_2 $ & No & 2 & 2- & No & Yes \\ \hline
$\mathbf{cmm} $ & $ D_2 $ & Yes & 2 & 2- & Yes & Yes \\ \hline
$\mathbf{p4} $ & $ C_4 $ & Yes & 2 & 2-,4- & No & No \\ \hline
$\mathbf{p4m} $ & $ D_4 $ & No & 2 & 2-,4- & Yes & Yes \\ \hline
$\mathbf{p4g} $ & $ D_4 $ & Yes & 2 & 2-,4- & Yes & Yes \\ \hline
$\mathbf{p3} $ & $ C_3 $ & Yes & 3 & 3- & No & No \\ \hline
$\mathbf{p3m1} $ & $ D_3 $ & Yes & 2,3 & 3- & Yes & Yes \\ \hline
$\mathbf{p31m} $ & $ D_3 $ & Yes & 2,3 & 3- & Yes & Yes \\ \hline
$\mathbf{p6} $ & $ C_6 $ & Yes & 2,3 & 2-,3-,6- & No & No \\ \hline
$\mathbf{p6m} $ & $ D_6 $ & Yes & 2,3 & 2-,3-,6- & Yes & Yes \\ \hline

\end{tabular}
\caption{Wallpaper groups}
\label{Wallpaper groups}
\end{center}
\end{table}

\newpage

\subsection{Classifying for proper actions}
\label{Classifying}

The main geometric object in this paper is the classifying space for proper actions, so we will recall the necessary definitions here. We refer the reader to \cite{Luc05} for a thorough exposition about the subject, and to \cite{TDieck} for generalities about group actions and $G$-CW-complexes.

\begin{definition}

Let $G$ a discrete group, $\mathcal{F}$ a family of subgroups closed under conjugation and subgroups. Then a $G$-CW-complex $E_{\mathcal{F}}G$ is called a \emph{classifying space for the family} $\mathcal{F}$ if given a subgroup $H<G$, the fixed point set $E_{\mathcal{F}}G^H$ is contractible if $H\in\mathcal{F}$ and empty otherwise.

\end{definition}

Observe that the definition implies that the stabilizers of the action of $G$ belong to $\mathcal{F}$. It can be proved that this space is unique up to $G$-homotopy equivalence.

When $\mathcal{F}$ is the trivial family, $E_{\mathcal{F}}G=EG$, the universal space for $G$-principal bundles. On the other hand, when $\mathcal{F}$ is the family of finite subgroups of $G$, $E_{\mathcal{F}}G$ is usually denoted by $\underline{E}G$, and called the  \emph{classifying space for proper actions of}  $G$. As stated in the introduction, this is the main object of interest of the left-hand side of Baum-Connes, and the computation of its Bredon homology for wallpaper groups $G$ is the main goal of this paper. In fact, these computations are feasible because there is a very simple model available for $E_{\mathcal{F}}G$ in this case:

\begin{theorem}

Let $G$ be a wallpaper group, and consider the usual action of $G$ on the plane via isometries. Then $\mathbb{R}^2$ is a model for $\underline{E}G$.

\end{theorem}

\begin{proof}

See \cite{ZVC80}, Section 4.2.

\end{proof}

The good knowledge of the actions of these groups of the plane will make possible the computation of the Bredon homology group. In next section we will recall the necessary definitions of this equivariant Homology Theory.

\subsection{Bredon homology}
\label{Bredon}

We recall in this section the main facts concerning Bredon homology, which is the main invariant we will deal with in this paper. We only review the topological version, following the approach of \cite{San08}; a good exposition that includes the algebraic version can be found in \cite{MiVa03}.

Let $G$ be discrete group, $\mathcal{F}$ a family of groups which is closed under conjugation and taking subgroups. Consider the \emph{orbit category} $O_{\mathcal{F}}(G)$, whose objects are the homogeneous spaces $G/K$, $K\subset G$ with $K\in\mathcal{F}$, and whose morphisms are the $G$-equivariant maps. Then a \emph{left Bredon module} $N$ over $O_{\mathcal{F}}(G)$ is a covariant functor $$N:O_{\mathcal{F}}(G)\ra \textbf{Ab},$$ where $\textbf{Ab}$ is the category of abelian groups.

Now consider a left Bredon module $N$ and a $G$-CW-complex $X$, and assume that all the stabilizers of the $G$-action belong to the family $\mathcal{F}$. Then the \emph{Bredon chain complex} $(C_n,\Phi_n)$ can be defined in the following way. For every $d\geq 0$, consider a set $\{e_i^d\}_{i\in I}$ of representatives of orbits of $d$-cells in $X$, and denote by $stab(e_i^d)$ the stabilizer of $e_i^d$. Then we define the \emph{n-th group of Bredon chains} as $C_d=\bigoplus_{i\in I} N(G/stab(e_i^d))$.

Now, consider a $(d-1)$-face of $e_i^d$, which can be given as $ge$ for a certain $(d-1)$-cell $e$. Then we have an inclusion of stabilizers $g^{-1}stab(e_i^d)g\subseteq stab(e)$. As $g^{-1}stab(e_i^d)g$ and $stab(e_i^d)$ are isomorphic, the previous inclusion induces an equivariant $G$-map $f:G/stab(e_i^d)\ra G/stab(e)$. In turn, as $N$ is a functor, we have an induced homomorphism $N(f):N(G/stab(e_i^d))\ra N(G/stab(e))$. Taking into account that the boundary of $e_i^d$ can be written as $\partial e^d_i=\sum_{j=1}^ne_j^{d-1}g_j$ for certain $g_j\in G$ and using linear extension to all representatives of equivariant $d$-cells, we obtain a differential $\Phi_d:C_d\ra C_{d-1}$ for every $d>0$. So have the following definition:

\begin{definition}
The homology groups of the chain complex $(C_i,\Phi_i)$ will be denoted by $H_i^{\mathcal{F}}(X,N)$ and called \emph{Bredon homology groups} of $G$ with coefficients in $N$ with respect to the family $\mathcal{F}$.
\end{definition}

These groups are an invariant of the $G$-homotopy type of $X$.

In this paper we are interested in Bredon homology with coefficients in the complex representation ring. In the next section we recall the definition of this coefficient module.

\subsection{Representation theory}
\label{Representation}

We refer the reader to the classic book of Serre \cite{Ser77} for all the basic concepts concerning complex representation theory of finite groups and their characters.

 As said above, we compute Bredon homology with respect to the family of the finite subgroups with coefficients in the representation ring Bredon module $\mathcal{R}_{\mathbb{C}}$. This module is defined in the following way. Given a group $G$, the functor $\mathcal{R}_{\mathbb{C}}:Or(G)\rightarrow \mathbf{Ab}$ is defined over objects as $\mathcal{R}_{\mathbb{C}}(G/K)=R_{\mathbb{C}}(K)$. the complex representation ring of $K$. To define the functor over morphisms, observe that for any equivariant map $f:G/K\ra G/H$ there exists $g\in G$ such that $gKg^{-1}\subseteq H$. As $R_{\mathbb{C}}(gKg^{-1})=R_{\mathbb{C}}(K)$, we can define $\mathcal{R}(f):R_{\mathbb{C}}(K)\ra R_{\mathbb{C}}(K)$ by induction from the subgroup inclusion $gKg^{-1}\subseteq H$. A detailed exposition about the properties of this functor can be found in \cite[Section 3]{MiVa03}.

To compute the differentials in the Bredon chain complex, we will need to know the homomorphisms between representation rings that are induced by inclusion of stabilizers in the wallpaper groups. In order to do so, we recall the structure of these rings as abelian groups and bases of irreducible characters in Table \ref{RTS}. There, the first element of each basis will always represent the trivial representation. In the case of $D_3$, which is isomorphic to the symmetric group $S_3$, $\chi_2$ stands for the sign representation and $\chi_3$ for the standard representation. For the dihedral groups, $\phi_i$ stand for the characters that correspond to 2-dimensional irreducible representations.

\begin{table}[t]
\begin{center}
\begin{tabular}{| c | c | c |}
\hline
Group & Representation ring & Basis \\ \hline
$C_n$ & $\Z^n$ & $\langle \chi_1,\ldots,\chi_n\rangle$ \\ \hline
$D_2$ & $\Z^4$ & $\langle \chi_1,\chi_2,\chi_3, \chi_4\rangle$ \\ \hline
$D_3$ & $\Z^3$ & $\langle \chi_1,\chi_2,\chi_3 \rangle$ \\ \hline
$D_4$ & $\Z^5$ & $\langle \chi_1,\chi_2,\chi_3,\chi_4, \phi \rangle$ \\ \hline
$D_6$ & $\Z^6$ & $\langle \chi_1,\chi_2,\chi_3,\chi_4, \phi_1, \phi_2 \rangle$ \\ \hline

\end{tabular}
\caption{Representation theory of stabilizers}
\label{RTS}
\end{center}
\end{table}

 Finally, in Table \ref{Stabilizers} we describe explicitly the homomorphisms between the representation rings of the stabilizers, which are easily obtained using the character tables of the groups and Frobenius reciprocity (see \cite{Ser77}, ch. 2). For the groups in the left-hand side of the homomorphisms (always the trivial group or $C_2$) the generators will be denoted by the letter $\rho$, being $\rho_1$ the trivial representation in the case of $C_2$. The notation $\rho\uparrow$ means, in each line, that we are giving the character induced by $\rho$ via the group inclusion of the left. In lines 10-17, $C_2^1$ is the conjugacy class of an order 2 element with non-trivial roots in $D_n$, while $C_2^2$ corresponds to an element with no non-trivial roots.  For the characters in the right-hand side we keep the notation of Table \ref{RTS}.

\begin{table}[t]
\begin{center}
\begin{tabular}{| c | c | c | c |}
\hline
 & Inclusion & Induced character & Image \\ \hline
1 & $\{1\}\hookrightarrow C_n$ & $\rho\uparrow$ & $\chi_1+\ldots+\chi_n $ \\ \hline
2 & $\{1\}\hookrightarrow D_2$ & $\rho\uparrow$ & $\chi_1+\chi_2+\chi_3+\chi_4 $ \\ \hline
3 & $\{1\}\hookrightarrow D_3$ & $\rho\uparrow$ & $\chi_1+\chi_2+2\chi_3 $ \\ \hline
4 & $\{C_2\}\hookrightarrow C_2$ & $\rho_1\uparrow$ & $\chi_1 $ \\ \hline
5 & $\{C_2\}\hookrightarrow C_2$ & $\rho_2\uparrow$ & $\chi_2 $ \\ \hline
6 & $\{C_2\}\hookrightarrow D_2$ & $\rho_1\uparrow$ & $\chi_1+\chi_2 $ \\ \hline
7 & $\{C_2\}\hookrightarrow D_2$ & $\rho_2\uparrow$ & $\chi_3+\chi_4 $ \\ \hline
8 & $\{C_2\}\hookrightarrow D_3$ & $\rho_1\uparrow$ & $\chi_1+\chi_3 $ \\ \hline
9 & $\{C_2\}\hookrightarrow D_3$ & $\rho_2\uparrow$ & $\chi_2+\chi_3 $ \\ \hline
10 & $\{C_2^1\}\hookrightarrow D_4$ & $\rho_1\uparrow$ & $\chi_1+\chi_2+\chi_3+\chi_4 $ \\ \hline
11 & $\{C_2^1\}\hookrightarrow D_4$ & $\rho_2\uparrow$ & $2\phi $ \\ \hline
12 & $\{C_2^2\}\hookrightarrow D_4$ & $\rho_1\uparrow$ & $\chi_1+\chi_3+\phi $ \\ \hline
13 & $\{C_2^2\}\hookrightarrow D_4$ & $\rho_2\uparrow$ & $\chi_2+\chi_4+\phi $ \\ \hline
14 & $\{C_2^1\}\hookrightarrow D_6$ & $\rho_1\uparrow$ & $\chi_1+\chi_2+2\phi_2 $ \\ \hline
15 & $\{C_2^1\}\hookrightarrow D_6$ & $\rho_2\uparrow$ & $\chi_3+\chi_4+2\phi_1 $ \\ \hline
16 & $\{C_2^2\}\hookrightarrow D_6$ & $\rho_1\uparrow$ & $\chi_1+\chi_3+\phi_1+\phi_2 $ \\ \hline
17 & $\{C_2^2\}\hookrightarrow D_6$ & $\rho_2\uparrow$ & $\chi_2+\chi_4+\phi_1+\phi_2 $ \\ \hline

\end{tabular}
\caption{Induced characters on stabilizers}
\label{Stabilizers}
\end{center}
\end{table}

\subsection{Smith normal form}
\label{SNF}

When computing the Bredon homology of wallpaper groups, the computation of the Smith normal form of a matrix is necessary to describe the groups and also for obtaining explicit bases, which is one of the main goals of this paper. A thorough treatment on the subject can be found in \cite{Haha70}, and we recall briefly here the main results that are used in the paper.

Let $A$ a $m\times n$ matrix with integer entries. Then there always exist invertible matrices $P$ and $Q$, of size $m$ and $n$ respectively, such that the matrix $D=PAQ$ has the following shape:

\begin{itemize}

\item For a certain $k\leq min(m,n)$ and for every $i\leq k$, the entries $d_{ii}$ of the matrix $D$ are nonzero integers.

\item For every $i\leq k$, $d_{ii}$ divides $d_{i+1,i+1}$.

\item The remaining entries of $D$ are zero.

\end{itemize}

The matrix $D$ is called the \emph{Smith normal form} of $A$ (usually abbreviated SNF) and is unique up to signs of the $d_{ii}$. The non-trivial entries are called the \emph{invariant factors} or the \emph{elementary divisors} of $A$.

We will use the Smith normal form to describe kernels and cokernels of homomorphisms between free abelian groups. Consider then a homomorphism $f:\Z^n\ra\Z^m$, and the associated matrix $A$ of size $m\times n$. Let $SNF(A)=PAQ$ a decomposition of the Smith normal form, and $(d_1, \dots ,d_k)$ the invariant factors. Then we have the following:

\begin{itemize}

\item Assume $|d_j|>1$, $|d_{j-1}|=1$, or $d_j=d_1$ if $|d_1|>1$.  Then the cokernel of $f$ is isomorphic to $\Z/(d_j)\times\ldots\times \Z/(d_k)\times \Z^{m-k}$.

\item The images of the last $m-k$ column vectors of $P^{-1}$ under the projection $\Z^m\ra Coker\textrm{ f}$ produce a basis of the torsion-free part of $Coker\textrm{ f}$.

\item The last $n-k$ column vectors of $Q$ provide a basis for the kernel of $f$.

\end{itemize}

These results will be essential when computing the Bredon homology groups.

\begin{remark}

There are different algorithms to compute the Smith normal form of a matrix. The computations for this paper have been performed used the algorithm implemented in \cite{Mat}. The outcome of our computations, including Smith normal forms and auxiliary matrices, is available on request.

\end{remark}

\section{Bredon homology of wallpaper groups}
\label{Computations}

In this section we undertake the main goal of this paper, which is the explicit computation of the Bredon homology of the wallpaper groups, with respect to the family of finite subgroups. Some notation will be required at this point. In general, for any of the wallpaper groups, a representative of a class of equivariant $i$-cells will be denoted by $e_i^j$. When there is only one equivariant $i$-cell the superscript will be suppressed. The irreducible characters in the representations rings of the stabilizers of $0$-cells will be denoted by $\alpha$, of stabilizers of $1$-cells by the letter $\beta$, and of stabilizers of $2$-cells by the letter $\gamma$. The homology classes in every chain group will be denoted by brackets. If $H<K$ is an inclusion of stabilizers and $\chi$ is a character on $H$, then the induced character on $K$ is denoted by $(\chi\uparrow K)$. In the exposition we will sometimes refer to Figure \ref{Doris} without express mention.

In the computation of the Bredon homology of every wallpaper group we undertake the following strategy. Starting from the pattern of the group of Figure \ref{Doris} and taking into account that the pattern contains a fundamental domain for the group, we describe a $G$-CW-complex structure in the plane with a unique class of equivariant 2-cells and we compute the boundaries of the 2-cells and 1-cells. Then, we describe the stabilizers of the cells and form the corresponding Bredon chain complex. After that, using the previously computed boundaries and the induced representations that are listed in Table \ref{Stabilizers} and taking account of the orientations, we describe the differentials of the Bredon complex. We conclude by using the Smith normal forms of the matrices of the differentials and their auxiliary matrices to describe the homology groups and their bases, in terms of irreducible characters of the stabilizers.

\subsection{The group $\mathbf{p1}$.}


 As this group is generated by the two translations, a representative $e_2$ for the equivariant 2-cells will be given by the polygon in Figure \ref{Doris}. Call $O$ the lowest vertex of the polygon $D$ in its left-hand side. Then the remaining vertices, going clockwise, are $P=g_1O$, $Q=g_2O$ and $R=g_3O$ for $g_1$ the vertical translation, $g_3$ the right translation and $g_2$ the product of them. We then consider just one class of $\mathbf{p1}$-equivariant $0$-cells in the plane, and we choose the representative $e_0$ corresponding to the vertex $O$. In this structure, there are two equivariant $1$-cells, which are represented by the edges $e_1^0=OP$ and $e_1^2=PQ$. Hence, the boundaries with respect to the orbit representatives are given by:

$$\partial e_2=e_1^0+e_1^1+g_3e_1^0+g_1e_1^1,$$
$$\partial e_1^0=g_1e_0-e_0,$$
$$\partial e_1^1=g_2e_0-g_1e_0.$$

As this group is torsion-free, all the stabilizers of the equivariant cells are trivial. Then the Bredon chain complex has the shape:

$$0\rightarrow \Z\gamma \rightarrow \Z\beta_1\oplus\Z\beta_2\rightarrow \Z\alpha\rightarrow 0.$$

Now we have:

$$\Phi_2(\gamma)=(\gamma \uparrow \textrm{stab}(e_1^0))+(\gamma \uparrow \textrm{stab}(e_1^1))-(\gamma \uparrow \textrm{stab}(e_1^0))-(\gamma \uparrow \textrm{stab}(e_1^1))=0,$$
$$\Phi_1(\beta_1)=(\beta_1 \uparrow \textrm{stab}(e_0^0))-(\beta_1 \uparrow \textrm{stab}(e_0^0))=0,$$
$$\Phi_1(\beta_2)=(\beta_2 \uparrow \textrm{stab}(e_0^0))-(\beta_2 \uparrow \textrm{stab}(e_0^0))=0.$$

Hence, the differentials in the Bredon complex are trivial, and we obtain $H_2^{\mathcal{F}}(\mathbf{p1},R_{\mathbb{C}})= \Z[\gamma]$, $H_1^{\mathcal{F}}(\mathbf{p1},R_{\mathbb{C}})=\Z[\beta_1]\oplus\Z[\beta_2]$ and $H_0^{\mathcal{F}}(\mathbf{p1},R_{\mathbb{C}})=\Z[\alpha]$.

Observe that, as $\mathbf{p1}$ is torsion-free, the Bredon homology groups of $\underline{E}{\mathbf{p1}}$ coincide with the ordinary homology groups of the classifying space $B\mathbf{p1}$. Hence, we have just recovered here the classical homology of the torus.

\subsection{The group $\mathbf{p2}$.}


Consider the lowest half of the polygon in Figure \ref{Doris}, which will be a representative for the equivariant 2-cell $e_2$. We consider the five vertices in this lowest half of the picture, and also the center of the polygon. Starting from the lowest left-hand side vertex and counting clockwise, we denote the vertices by $O$, $P$, $Q$, $R$, $S$ and $T$. We consider four classes of $0$-cells, with representatives $e_0^0$, $e_0^1$, $e_0^2$ and $e_0^3$ that correspond respectively to the vertices $O$, $P$, $Q$ and $T$. Observe that $R$ is the image of $P$ under the rotation $r_1$ of center $Q$, and $S$ is the image of $O$ under the rotation $r_2$ of center $T$. In turn, there are three classes of $1$-cells, with representatives $e_1^0$, $e_1^1$ and $e_1^2$ that correspond to the edges $OP$, $PQ$ and $ST$. Hence, if $t$ is the horizontal translation, we have the boundaries:

$$\partial(e_2)=e_1^0+e_1^1+r_1e_1^1+te_1^0+e_1^2+r_2e_1^2,$$
$$\partial(e_1^0)=e_0^1-e_0^0,$$
$$\partial(e_1^1)=e_0^2-e_0^1,$$
$$\partial(e_1^2)=e_0^3-e_0^0.$$

The only nontrivial stabilizers correspond to the 0-cells (which are 2-centers of rotation), and are all isomorphic to $C_2.$  Then the chain complex is:

$$0\ra \Z\gamma \ra \bigoplus_{i=0}^2\Z\beta_i\ra \bigoplus_{i=0}^3 (\Z\alpha_i^1\oplus\Z\alpha_i^2)\ra 0.$$

Taking account of line 1 in Table \ref{Stabilizers}, the differentials are defined in the following way:

$$\Phi_2(\gamma)=(\gamma \uparrow \textrm{stab}(e_1^0))+(\gamma \uparrow \textrm{stab}(e_1^1))-(\gamma \uparrow \textrm{stab}(e_1^1))-(\gamma \uparrow \textrm{stab}(e_1^0))+(\gamma \uparrow \textrm{stab}(e_1^2))-(\gamma \uparrow \textrm{stab}(e_1^2))=0,$$
$$\Phi_1(\beta_0)=(\beta_0 \uparrow \textrm{stab}(e_0^1))-(\beta_1 \uparrow \textrm{stab}(e_0^0))=\alpha^1_1+\alpha^2_1-\alpha^1_0-\alpha^2_0,$$
$$\Phi_1(\beta_1)=(\beta_1 \uparrow \textrm{stab}(e_0^2))-(\beta_2 \uparrow \textrm{stab}(e_0^1))=\alpha^1_2+\alpha^2_2-\alpha^1_1-\alpha^2_1,$$
$$\Phi_1(\beta_2)=(\beta_2 \uparrow \textrm{stab}(e_0^3))-(\beta_2 \uparrow \textrm{stab}(e_0^0))=\alpha^1_3+\alpha^2_3-\alpha^1_0-\alpha^2_0.$$

Now computing the Smith normal form of the matrix of $\Phi_1$ we obtain that the invariant factors of $\Phi_1$ are $(1,1,1)$. As $\Phi_2$ is trivial, this implies that $H_2^{\mathcal{F}}(\mathbf{p2},R_{\mathbb{C}})= \Z$, $H_1^{\mathcal{F}}(\mathbf{p2},R_{\mathbb{C}})=0$ and $H_0^{\mathcal{F}}(\mathbf{p2},R_{\mathbb{C}})=\Z^5$.

It is clear that a basis for $H_2^{\mathcal{F}}(\mathbf{p2},R_{\mathbb{C}})$ is $[\gamma]$. In turn, the matrix $Q$ obtained in the computation of the SNF shows that a basis of $H_0^{\mathcal{F}}(\mathbf{p2},R_{\mathbb{C}})$ is given by $([\alpha_0^1], [\alpha_0^2], [\alpha_1^2], [\alpha_2^2], [\alpha_3^2])$.

\subsection{The group $\mathbf{pm}$.}


Again in this case a representative $e_2$ for the unique 2-cell will be given by the lowest half of the rectangle. Consider $O$, $P$, $Q$ and $R$ the vertices of this little rectangle, starting from the lowest in the left-hand side edge. There will be two classes of $0$-cells, with representatives $e_0^0$ and $e_0^1$, that correspond respectively to the vertices $O$ and $P$, with $Q$ in the class of $P$ and $R$ in the class of $O$ (the identification given by the horizontal translation $t$). In turn, there are two classes of $1$-cells, with representatives $e_1^0$, $e_1^1$ and $e_1^2$, given by the edges $OP$, $PQ$ and $RP$ respectively. Observe that $RQ$ is the image of $OP$ under $t$. Now we can compute the boundaries:

$$\partial(e_2)=e_1^0+e_1^1+te_1^0+e_1^2,$$
$$\partial(e_1^0)=e_0^1-e_0^0,$$
$$\partial(e_1^1)=te_0^1-e_0^1,$$
$$\partial(e_1^2)=e_0^0-te_0^0.$$

The vertices and the edges $e_1^1$ and $e_1^2$ are in rotation axes, so their stabilizers are isomorphic to $C_2$, while the stabilizer of the remaining edge is trivial. The the Bredon chain complex takes the following shape:

$$0\ra \Z\gamma \ra \Z\beta_0\oplus\Z\beta^1_1\oplus\Z\beta^2_1\oplus\Z\beta^1_2\oplus\Z\beta^2_2\  \ra \bigoplus_{i=0}^1 (\Z\alpha_i^1\oplus\Z\alpha_i^2)\ra 0.$$

Now we can compute the differentials, taking account of line 1 in Table \ref{Stabilizers}:

$$\Phi_2(\gamma)=(\gamma \uparrow \textrm{stab}(e_1^0))+(\gamma \uparrow \textrm{stab}(e_1^1))-(\gamma \uparrow \textrm{stab}(e_1^0))+(\gamma \uparrow \textrm{stab}(e_1^2))=\beta_1^1+\beta_1^2+\beta_2^1+\beta_2^2,$$
$$\Phi_1(\beta_0)=(\beta_0 \uparrow \textrm{stab}(e_0^1))-(\beta_1 \uparrow \textrm{stab}(e_0^0))=\alpha_1^1+\alpha_1^2-\alpha_0^1-\alpha_0^2,$$
$$\Phi_1(\beta^1_1)=(\beta^1_1 \uparrow \textrm{stab}(e_0^1))-(\beta^1_1 \uparrow \textrm{stab}(e_0^1))=0,$$
$$\Phi_1(\beta^2_1)=(\beta^2_1 \uparrow \textrm{stab}(e_0^1))-(\beta^2_1 \uparrow \textrm{stab}(e_0^1))=0,$$
$$\Phi_1(\beta^1_2)=(\beta^1_2 \uparrow \textrm{stab}(e_0^0))-(\beta^1_1 \uparrow \textrm{stab}(e_0^0))=0,$$
$$\Phi_1(\beta^2_2)=(\beta^2_2 \uparrow \textrm{stab}(e_0^0))-(\beta^2_1 \uparrow \textrm{stab}(e_0^0))=0.$$

 We compute the SNF of the matrices of $\Phi_2$ and $\Phi_1$ and we respectively obtain that the invariant factors are $(1)$ and $(1)$. This implies that $H_2^{\mathcal{F}}(\mathbf{pm},R_{\mathbb{C}})= 0$, $H_1^{\mathcal{F}}(\mathbf{pm},R_{\mathbb{C}})=\Z^3$ and $H_0^{\mathcal{F}}(\mathbf{pm},R_{\mathbb{C}})=\Z^3$.

Now, the matrix $Q$ obtained in the computation of the SNF for $\Phi_1$ and the definition of $\Phi_2$ show that a basis of $H_1^{\mathcal{F}}(\mathbf{pm},R_{\mathbb{C}})$ is given by $([\beta_1^1], [\beta_1^2], [\beta_2^1])$, while a basis for $H_0^{\mathcal{F}}(\mathbf{pm},R_{\mathbb{C}})$ is given by $([\alpha_0^2], [\alpha_1^1], [\alpha_1^2]).$

\subsection{The group $\mathbf{pg}$.}


Here we divide the (big) rectangle in Figure \ref{Doris} in two equal rectangles by a vertical line; then the left one will be a representative $e_2$ of the unique class of $2$-cells. Consider the vertices $O,P,Q,R$ of this rectangle, counting clockwise from the left-hand lowest vertex $O$. Observe that $P$ is the image of $O$ under vertical translation, $Q$ under glide-reflection, and $R$ under the composition of both. Hence, there will also be a unique class of $0$-cells, and we denote by $e_0$ the representative given by $O$. In turn, there are two classes of 1-cells, identified by $OP$ and $PQ$, which we respectively denote $e_1^0$ and $e_1^1$. Hence, if we call $t$ the vertical translation (going upwards) and $g$ the glide reflection, the boundaries of the representatives are defined in the following way:

$$\partial(e_2)=e_1^0+e_1^1+ge_1^0+t^{-1}e_1^1,$$
$$\partial(e_1^0)=te_0^0-e_0^0,$$
$$\partial(e_1^1)=ge_0^0-te_0^0.$$

The group $\mathbf{pg}$ is torsion-free, and hence all its stabilizers are trivial. This time the chain complex is quite simple:

$$0\ra \Z\gamma \ra \Z\beta_0\oplus\Z\beta_1\ra \Z\alpha.$$

Now we compute the differentials of the complex:

$$\Phi_2(\gamma)=(\gamma \uparrow \textrm{stab}(e_1^0))+(\gamma \uparrow \textrm{stab}(e_1^1))+(\gamma \uparrow \textrm{stab}(e_1^0))-(\gamma \uparrow \textrm{stab}(e_1^1))=2\beta_0,$$
$$\Phi_1(\beta_0)=(\beta_0 \uparrow \textrm{stab}(e_0^0))-(\beta_0 \uparrow \textrm{stab}(e_0^0))=0,$$
$$\Phi_1(\beta^1)=(\beta_1 \uparrow \textrm{stab}(e_0^0))-(\beta^1 \uparrow \textrm{stab}(e_0^0))=0.$$

As $\Phi_1$ is trivial, $H_0^{\mathcal{F}}(\mathbf{pg},R_{\mathbb{C}})=\Z$. On the other hand, the invariant factor of the SNF of the matrix of $\Phi_2$ is $(2)$, so $H_1^{\mathcal{F}}(\mathbf{pg},R_{\mathbb{C}})=\Z\oplus\Z /2$ and $H_2^{\mathcal{F}}(\mathbf{pg},R_{\mathbb{C}})=0$.

By construction it is easy to see here that a basis for $H_0^{\mathcal{F}}(\mathbf{pg},R_{\mathbb{C}})$ is given by $[\alpha]$, while $[\beta_0]$ and $[\beta_1]$ generate respectively the torsion part and the free part of $H_1^{\mathcal{F}}(\mathbf{pg},R_{\mathbb{C}})$.

As the group is torsion-free, we recover again (as in the case of $\mathbf{p1}$) the ordinary homology of the classifying space $B\mathbf{pg}$, which in this case has the homotopy type of the Klein bottle.

\subsection{The group $\mathbf{cm}$.}


Here our representative $e_2$ of the class of equivariant $2$-cells will be given by the lowest half of the rhombus. We denote its vertices by $O$, $P$ and $Q$, starting from the one in the left and counting anti-clockwise. The horizontal translation $t$ takes $O$ to $Q$, while a glide-reflection sends $O$ to $P$ and $P$ to $Q$. Then, again we consider a unique class of equivalence of $0$-cells, whose representative $e_0$ is identified with $O$. There are also two classes of $1$-cells, whose representatives $e_1^0$ and $e_1^1$ we identify with $OP$ and $QR$ (observe that $PQ$ is the image of $OP$ under the glide-reflection $g$). Now the boundaries are given by:

$$\partial(e_2)=e_1^0+ge_1^0+e_1^1,$$
$$\partial(e_1^0)=ge_0^0-e_0^0,$$
$$\partial(e_1^1)=e_0^0-g^2e_0^0.$$

Observe that both $e_0^0$ and $e_1^1$ lie in a reflection axis, so their stabilizers are isomorphic to $C_2$. On the other hand, the stabilizer of the other edge is trivial, and hence we have the following Bredon chain complex:

$$0\ra \Z\gamma \ra \Z\beta_0\oplus \Z\beta_1^1\oplus \Z\beta_1^2\ra \Z\alpha^1\oplus \Z\alpha^2\ra 0.$$

The differentials are quite simple in this case, taking again account of line 1 in Table \ref{Stabilizers}:

$$\Phi_2(\gamma)=(\gamma \uparrow \textrm{stab}(e_1^0))+(\gamma \uparrow \textrm{stab}(e_1^0))+(\gamma \uparrow \textrm{stab}(e_1^1))=2\beta_0+\beta_1^1+\beta_1^2,$$
$$\Phi_1(\beta_0)=(\beta_0 \uparrow \textrm{stab}(e_0^0))-(\beta_0 \uparrow \textrm{stab}(e_0^0))=0,$$
$$\Phi_1(\beta^1)=(\beta_1 \uparrow \textrm{stab}(e_0^0))-(\beta^1 \uparrow \textrm{stab}(e_0^0))=0.$$

Again the triviality of $\Phi_1$ immediately implies $H_0^{\mathcal{F}}(\mathbf{cm},R_{\mathbb{C}})=\Z^2$. On the other hand, the unique invariant factor of the SNF of the matrix of $\Phi_2$ is $(1)$, so $H_1^{\mathcal{F}}(\mathbf{cm},R_{\mathbb{C}})=\Z^2$ and $H_2^{\mathcal{F}}(\mathbf{cm},R_{\mathbb{C}})=0$.

It is clear that a basis of $H_0^{\mathcal{F}}(\mathbf{cm},R_{\mathbb{C}})$ is given by $[\alpha^1]$ and $[\alpha^2]$, while the definition of $\Phi_2$ implies that $[\beta_1^1]$ and $[\beta_1^2]$ form a basis for $H_1^{\mathcal{F}}(\mathbf{cm},R_{\mathbb{C}})$.

\subsection{The group $\mathbf{pmm}$.}


Here the lowest left small square can be taken a representative $e_2$ of the equivariant class of 2-cells under the action of $\mathbf{pmm}$. We consider the four vertices $O$, $P$, $Q$ and $R$ of this small square, starting as always in the lowest one of the left-hand side, and counting clockwise. Each of this edges will correspond respectively to a representatives of different classes of 0-cells, say $e_0^0,$ $e_0^1$, $e_0^2$ and $e_0^3$. In turn, there will be also be four representatives of classes of $1$-cells, namely $e_1^0,$ $e_1^1$, $e_1^2$ and $e_1^3$, which we respectively identify with the edges $OP$, $PQ$, $QR$ and $RP.$ The boundaries in this case are easy, because the group makes no identifications inside the small square:

$$\partial(e_2)=e_1^0+e_1^1+e_1^2+e_1^3,$$
$$\partial(e_1^0)=e_0^1-e_0^0,$$
$$\partial(e_1^1)=e_0^2-e_0^1,$$
$$\partial(e_1^2)=e_0^3-e_0^2,$$
$$\partial(e_1^3)=e_0^0-e_0^3.$$

In this model all the edges lie on reflection axes, and there no other relevant isometries. Hence, all the stabilizers of the edges are isomorphic to $C_2$. In turn, every vertex lie in two different reflection axes, so the stabilizers of the vertices are isomorphic to $D_2$. The Bredon complex takes then the following shape:

$$0\ra \Z\gamma \ra \bigoplus_{i=0}^3 (\beta_i^1\oplus \beta_i^2) \ra \bigoplus_{i=0}^3 (\alpha_i^1\oplus \alpha_i^2\oplus \alpha_i^3\oplus \alpha_i^4)\ra 0.$$

Taking account of lines 1, 6 and 7 in Table \ref{Stabilizers}, we compute the differentials:

$$\Phi_2(\gamma)=(\gamma \uparrow \textrm{stab}(e_1^0))+(\gamma \uparrow \textrm{stab}(e_1^1))+(\gamma \uparrow \textrm{stab}(e_1^2)+(\gamma \uparrow \textrm{stab}(e_1^3)=\sum_{i=0}^3\sum_{j=1}^2 \beta_i^j,$$
$$\Phi_1(\beta_0^1)=(\beta_0^1 \uparrow \textrm{stab}(e_0^1))-(\beta_0 \uparrow \textrm{stab}(e_0^0))=\alpha^1_1+\alpha^2_1-\alpha^1_0-\alpha^2_0,$$
$$\Phi_1(\beta_0^2)=(\beta_0^2 \uparrow \textrm{stab}(e_0^1))-(\beta_0 \uparrow \textrm{stab}(e_0^0))=\alpha^3_1+\alpha^4_1-\alpha^3_0-\alpha^4_0,$$
$$\Phi_1(\beta_1^1)=(\beta_1^1 \uparrow \textrm{stab}(e_0^2))-(\beta_0 \uparrow \textrm{stab}(e_0^1))=\alpha^1_2+\alpha^2_2-\alpha^1_1-\alpha^3_1,$$
$$\Phi_1(\beta_1^2)=(\beta_1^2 \uparrow \textrm{stab}(e_0^2))-(\beta_0 \uparrow \textrm{stab}(e_0^1))=\alpha^3_2+\alpha^4_2-\alpha^2_1-\alpha^4_1,$$
$$\Phi_1(\beta_2^1)=(\beta_2^1 \uparrow \textrm{stab}(e_0^3))-(\beta_0 \uparrow \textrm{stab}(e_0^2))=\alpha^1_3+\alpha^2_3-\alpha^1_2-\alpha^3_2,$$
$$\Phi_1(\beta_2^2)=(\beta_2^2 \uparrow \textrm{stab}(e_0^3))-(\beta_0 \uparrow \textrm{stab}(e_0^2))=\alpha^3_3+\alpha^4_3-\alpha^2_2-\alpha^4_2,$$
$$\Phi_1(\beta_3^1)=(\beta_3^1 \uparrow \textrm{stab}(e_0^0))-(\beta_0 \uparrow \textrm{stab}(e_0^3))=\alpha^1_0+\alpha^3_0-\alpha^1_3-\alpha^3_3,$$
$$\Phi_1(\beta_3^2)=(\beta_3^2 \uparrow \textrm{stab}(e_0^0))-(\beta_0 \uparrow \textrm{stab}(e_0^3))=\alpha^2_0+\alpha^4_0-\alpha^2_3-\alpha^4_3.$$

Observe that when we described the differentials, we have taken into account that two coincident reflection edges define different subgroups (isomorphic to $C_2$) in the stabilizer of the common vertex.

We now compute the SNF of the matrices of $\Phi_2$ and $\Phi_1$ and we obtain that the invariant factors are $(1)$ and $(1,1,1,1,1,1,1)$, respectively. This implies that $H_2^{\mathcal{F}}(\mathbf{pmm},R_{\mathbb{C}})= 0$, $H_1^{\mathcal{F}}(\mathbf{pmm},R_{\mathbb{C}})=0$ and $H_0^{\mathcal{F}}(\mathbf{pmm},R_{\mathbb{C}})=\Z^9$.

Finally, the last columns of the matrix $Q$ obtained in the computation of the SNF for $\Phi_1$  show that a basis for $H_0^{\mathcal{F}}(\mathbf{pm},R_{\mathbb{C}})$ is given by $([\alpha_0^3], [\alpha_1^3], [\alpha_1^4], [\alpha_2^3], [\alpha_2^4], [\alpha_3^1], [\alpha_3^2], [\alpha_3^3], [\alpha_3^4])$.

\subsection{The group $\mathbf{pmg}$.}


A representative $e_2$ for the class of equivariant 2-cells in $\mathbf{pmg}$ will given by any of the two rectangles of the picture in Figure \ref{Doris} whose horizontal edges are reflection axes, so we choose for example the left one. We consider six vertices on it: the four vertices given by the corners of the rectangle, and the marked rotation centers in the middle points of the vertical sides. Starting from the lowest vertex of the left-hand side of the rectangle and going clockwise, we call these vertices $O$, $P$, $Q$, $R$, $S$ and $T$. Then, representatives of the four classes $e_0^0$, $e_0^1$, $e_0^2$ and $e_0^3$ of equivariant 0-cells will be respectively given by the vertices $O$, $P$, $R$ and $S$. Remark that $Q$ is the image of $O$ under the rotation $r_1$ of center $P$, and $T$ is the image of $R$ under the rotation $r_2$ of center $S$. There will also four classes of equivariant 1-cells, whose representatives $e_1^0$, $e_1^1$, $e_1^2$ and $e_1^3$ are identified with $OP$, $QR$, $RS$ and $TP$. Observe that $r_1(OP)=QP$ and $r_2(RS)=TS$. We are now ready to compute the boundaries:

$$\partial(e_2)=e_1^0+r_1e_1^0+e_1^1+e_1^2+r_2e_2+e_1^3,$$
$$\partial(e_1^0)=e_0^1-e_0^0,$$
$$\partial(e_1^1)=e_0^2-r_1e_0^0,$$
$$\partial(e_1^2)=e_0^3-e_0^2,$$
$$\partial(e_1^3)=e_0^0-r_2e_0^2.$$

Now, the horizontal edges of the rectangle are in reflection axes, and same happens to the vertices $e_0^0$ and $e_0^2$, so the stabilizers of the corresponding cells are isomorphic to $C_2$. As $e_0^1$ and $e_0^3$ are centers of 2-rotation the stabilizers are also isomorphic to $C_2$. Finally, the group acts freely over the classes of $e_1^0$ and $e_1^0$, so we can form the Bredon chain complex:

$$0\ra \Z\gamma \ra \Z\beta_0\oplus\Z\beta_1^1\oplus\Z\beta_1^2\oplus\Z\beta_2^1\oplus\Z\beta_2^2\oplus\Z\beta_3\ra \bigoplus_{i=1}^2\Z\alpha^i_0\oplus\bigoplus_{i=1}^2\Z\alpha^i_1\bigoplus_{i=1}^2\Z\alpha^i_2\bigoplus_{i=1}^2\Z\alpha^i_3\ra 0.$$

Let us compute the differentials of the complex, taking account of line 1 in Table \ref{Stabilizers}:

$$\Phi_2(\gamma)=(\gamma \uparrow \textrm{stab}(e_1^0))-(\gamma \uparrow \textrm{stab}(e_1^0))+(\gamma \uparrow \textrm{stab}(e_1^1)+(\gamma \uparrow \textrm{stab}(e_1^2)-(\gamma \uparrow \textrm{stab}(e_1^2))+(\gamma \uparrow \textrm{stab}(e_1^3))=$$ $$=\beta_1^1+\beta_1^2+\beta_2^1+\beta_2^2,$$
$$\Phi_1(\beta_0)=(\beta_0 \uparrow \textrm{stab}(e_0^1))-(\beta_0 \uparrow \textrm{stab}(e_0^0))=\alpha^1_1+\alpha^2_1-\alpha^1_0-\alpha^2_0,$$
$$\Phi_1(\beta_1^1)=(\beta_0^2 \uparrow \textrm{stab}(e_0^2))-(\beta_1^1 \uparrow \textrm{stab}(e_0^0))=\alpha^1_2+\alpha^1_0,$$
$$\Phi_1(\beta_1^2)=(\beta_1^2 \uparrow \textrm{stab}(e_0^2))-(\beta_1^2 \uparrow \textrm{stab}(e_0^0))=\alpha^2_2+\alpha^2_0,$$
$$\Phi_1(\beta_2^1)=(\beta_2^1 \uparrow \textrm{stab}(e_0^3))-(\beta_2^1 \uparrow \textrm{stab}(e_0^2))=\alpha^1_3+\alpha^2_3-\alpha^1_2-\alpha^2_2,$$
$$\Phi_1(\beta_2^2)=(\beta_2^2 \uparrow \textrm{stab}(e_0^3))-(\beta_2^2 \uparrow \textrm{stab}(e_0^2))=\alpha^1_0+\alpha^1_2,$$
$$\Phi_1(\beta_3)=(\beta_3 \uparrow \textrm{stab}(e_0^0))-(\beta_3 \uparrow \textrm{stab}(e_0^2))=\alpha^2_0+\alpha^2_2.$$

 From the SNF of the matrices of $\Phi_2$ and $\Phi_1$ it is obtained that the invariant factors are $(1)$ and $(1,1,1,1)$, respectively. This implies that $H_2^{\mathcal{F}}(\mathbf{pmg},R_{\mathbb{C}})= 0$, $H_1^{\mathcal{F}}(\mathbf{pmg},R_{\mathbb{C}})=\Z$ and $H_0^{\mathcal{F}}(\mathbf{pmg},R_{\mathbb{C}})=\Z^4$.

In turn, the matrices $P$ and $Q$ obtained in the computation of the Smith normal forms show that a basis for $H_0^{\mathcal{F}}(\mathbf{pmg},R_{\mathbb{C}})$ is given $([\alpha_1^2], [\alpha_2^1], [\alpha_2^2], [\alpha_3^2])$ and a basis for $H_1^{\mathcal{F}}(\mathbf{pmg},R_{\mathbb{C}})$ is given by $[\beta_1^1+\beta_2^1]$.

\subsection{The group $\mathbf{pgg}$.}


For this group, a representative $e_2$ for the class of equivariant 2-cells will be given for example by the triangle determined by the middle points of the vertical sides and the center of the lowest horizontal side of the (big) rectangle of the picture in Figure \ref{Doris}. Consider these three vertices and the center of the rectangle, and call them $O$, $P$, $Q$ and $R$, starting from the middle point of the left vertical side and counting clockwise. Then representatives $e_0^0$ and $e_0^1$ for the classes of $0$-cells will be given by $O$ and $P$, being $Q$ the image of $O$ under a rotation $r$ of center $P$ and $R$ the image of $O$ under a glide-reflection $g$. Representatives $e_1^0$ and $e_1^1$ for the classes of 1-cells are given by $RO$ and $OP$ respectively, being $QP=r(OP)$ and $RQ=g(OR)$. Now the boundaries are defined in the following way:

$$\partial(e_2)=e_1^1+re_1^1+ge_1^0+e_1^0,$$
$$\partial(e_1^0)=ge_0^0-e_0^0,$$
$$\partial(e_1^1)=e_0^1-e_0^0.$$

The two representatives of $0$-cells are centers of 2-rotation, and then their stabilizers are isomorphic to $C_2$. On the other hand, the group acts freely over the classes of $e_1^0$ and $e_1^1$, so we obtain the following Bredon complex:

$$0\ra \Z\gamma \ra \Z\beta_0\oplus\Z\beta_1\ra \Z\alpha_0^1\oplus\Z\alpha_0^2\oplus\Z\alpha_1^1\oplus\Z\alpha_1^2.$$

The differentials of the chain complex are now given in the following way, taking account of line 1 in Table \ref{Stabilizers}:

$$\Phi_2(\gamma)=(\gamma \uparrow \textrm{stab}(e_1^1))-(\gamma \uparrow \textrm{stab}(e_1^1))+(\gamma \uparrow \textrm{stab}(e_1^0))+(\gamma \uparrow \textrm{stab}(e_1^0))=2\beta_1,$$
$$\Phi_1(\beta_0)=(\beta_0 \uparrow \textrm{stab}(e_0^1))-(\beta_0 \uparrow \textrm{stab}(e_0^0))=\alpha^1_1+\alpha^2_1-\alpha^1_0-\alpha^2_0,$$
$$\Phi_1(\beta_1)=(\beta_1 \uparrow \textrm{stab}(e_0^0))-(\beta_1 \uparrow \textrm{stab}(e_0^0))=0.$$

 Now the Smith normal form of the matrices of $\Phi_2$ and $\Phi_1$ gives respectively the invariant factors $(2)$ and $(1)$. Hence, we have $H_2^{\mathcal{F}}(\mathbf{pgg},R_{\mathbb{C}})= 0$, $H_1^{\mathcal{F}}(\mathbf{pgg},R_{\mathbb{C}})=\Z /2$ and $H_0^{\mathcal{F}}(\mathbf{pgg},R_{\mathbb{C}})=\Z^3$.

 Also, the matrices $P$ and $Q$ obtained in the computation of the Smith normal forms show that a basis for $H_0^{\mathcal{F}}(\mathbf{pgg},R_{\mathbb{C}})$ is given by $([\alpha_1^2], [\alpha_2^1], [\alpha_2^2])$ and a generator for $H_1^{\mathcal{F}}(\mathbf{pgg},R_{\mathbb{C}})$ is given by $[\beta_0]$.

\subsection{The group $\mathbf{cmm}$.}


Consider the four triangles inside the rhombus in the figure. A representative $e_2$ for the equivalence class of 2-cells will be given by the lowest left-hand side triangle. Starting from the left and going clockwise, denote by $O$, $P$ and $Q$ the vertices of this triangle, and by $R$ the middle point of the diagonal side. Then there will be three representatives $e_0^0$, $e_0^1$ and $e_0^2$ of the classes of 0-cells, corresponding respectively to the vertices $O$ and $P$ and $R$; observe that $Q$ is the image of $P$ under the 2-rotation $r$ whose center is $R$. In turn, there are three classes of 1-cells, with representatives $e_1^0$, $e_1^1$ and $e_1^2$ identified respectively with $OP$, $PQ$ and $QR$. For the remaining edge we have $r(QR)=RP$. Let us describe now the boundaries:

$$\partial(e_2)=e_1^0+e_1^1+e_1^2+re_1^2,$$
$$\partial(e_1^0)=e_0^1-e_0^0,$$
$$\partial(e_1^1)=re_0^0-e_0^0,$$
$$\partial(e_1^2)=e_0^2-re_0^0.$$

Now, in both $e_0^0$ and $e_0^1$ two reflection axes cross, an hence the stabilizers of these two cells are isomorphic to $D_2$. As $e_0^2$ is a center of 2-rotation, its stabilizer is $C_2$. Concerning the 1-cells, $e_1^0$ and $e_1^1$ are in reflection axes, so their stabilizer is $C_2$, while the group act freely over the class of $e_1^2$. Hence we have the chain complex:

$$0\ra \Z\beta_0^1\oplus\Z\beta_0^2\oplus\Z\beta_1^1\oplus\Z\beta_1^2\oplus\Z\beta_2\ra \bigoplus_{i=1}^4\Z\alpha_0^i\bigoplus_{i=1}^4\Z\alpha_1^i\bigoplus_{i=1}^2\Z\alpha_2^i\ra 0.$$

Taking account of lines 1, 2, 6 and 7 in Table \ref{Stabilizers}, the differentials in this case are given by:

$$\Phi_2(\gamma)=(\gamma \uparrow \textrm{stab}(e_1^0))+(\gamma \uparrow \textrm{stab}(e_1^1))+(\gamma \uparrow \textrm{stab}(e_1^2))-(\gamma \uparrow \textrm{stab}(e_1^2))=\beta_0^1+\beta_0^2+\beta_1^1+\beta_1^2,$$
$$\Phi_1(\beta_0^1)=(\beta_0^1 \uparrow \textrm{stab}(e_0^1))-(\beta_0^1 \uparrow \textrm{stab}(e_0^0))=\alpha^1_1+\alpha^2_1-\alpha^1_0-\alpha^2_0,$$
$$\Phi_1(\beta_0^2)=(\beta_0^2 \uparrow \textrm{stab}(e_0^1))-(\beta_0^2 \uparrow \textrm{stab}(e_0^0))=\alpha^3_1+\alpha^4_1-\alpha^3_0-\alpha^4_0,$$
$$\Phi_1(\beta_1^1)=(\beta_1^1 \uparrow \textrm{stab}(e_0^0))-(\beta_1^1 \uparrow \textrm{stab}(e_0^1))=\alpha^1_0+\alpha^3_0-\alpha^1_1-\alpha^3_1,$$
$$\Phi_1(\beta_1^2)=(\beta_1^2 \uparrow \textrm{stab}(e_0^0))-(\beta_1^2 \uparrow \textrm{stab}(e_0^1))=\alpha^2_0+\alpha^4_0-\alpha^2_1-\alpha^4_1.$$
$$\Phi_1(\beta_2)=(\beta_2 \uparrow \textrm{stab}(e_0^2))-(\beta_2 \uparrow \textrm{stab}(e_0^0))=\alpha^1_2+\alpha^2_2-\alpha^1_0-\alpha^2_0-\alpha^3_0-\alpha^4_0.$$

The SMF of the matrices of $\Phi_2$ and $\Phi_1$ give the invariant factors $(2)$ and $(1,1,1,1)$, respectively. This implies $H_2^{\mathcal{F}}(\mathbf{cmm},R_{\mathbb{C}})= 0$, $H_1^{\mathcal{F}}(\mathbf{cmm},R_{\mathbb{C}})=0$ and $H_0^{\mathcal{F}}(\mathbf{cmm},R_{\mathbb{C}})=\Z^6$.

Using again the matrix $Q$ associated to the SMF, we obtain that a basis of $H_0^{\mathcal{F}}(\mathbf{cmm},R_{\mathbb{C}})$ is given by ($[\alpha_0^1+\alpha_0^2], [\alpha_0^3], [\alpha_0^4], [\alpha_1^1], [\alpha_1^3], [\alpha_2^2]$).

\subsection{The group $\mathbf{p4}$.}


We divide the square in the picture in four equal squares, using the vertical segment defined by the middle points of the horizontal sides and the horizontal segment defined by the middle points of the vertical sides. The representative $e_2$ for the class of 2-equivariant cells will be the lowest little square in the left-hand side. Denote by $O$, $P$, $Q$ and $R$ the vertices of this little square, starting on the down-left and going clockwise. There are three classes of 0-cells, whose representatives $e_0^0$, $e_0^1$ and $e_0^2$ correspond respectively to $O$, $P$ and $Q$; observe that if $t$ is the 4-rotation (counterclockwise) whose rotation center is $Q$, $t(P)=R$. Moreover, there are two classes of 1-cells, with representatives $e_1^0$ and $e_1^1$ corresponding to the sides $OP$ and $PQ$; the two remaining sides can be obtained as $RQ=t(PQ)$ and $OR=t(OP)$. We have the boundaries:

$$\partial(e_2)=e_1^0+e_1^1+te_1^1+te_1^0,$$
$$\partial(e_1^0)=e_0^1-e_0^0,$$
$$\partial(e_1^1)=e_0^2-e_0^1.$$

Concerning the stabilizers, $e_0^0$ and $e_0^2$ are centers of 4-rotation, and hence their stabilizers are isomorphic to $C_4$; while $e_0^1$ is a center of 2-rotation, and its stabilizer is $C_2$. The group $\mathbf{p4}$ acts freely on each class of 1-cells, so the Bredon chain complex has the following shape:

$$0\ra \Z\gamma\ra \Z\beta_0\oplus\Z\beta_1 \ra \bigoplus_{i=1}^4\Z\alpha_0^i\bigoplus_{i=1}^2\Z\alpha_1^i\bigoplus_{i=1}^4\Z\alpha_2^i\ra 0.$$

Now we can describe the differentials of the complex, taking account of line 1 in Table \ref{Stabilizers}:

$$\Phi_2(\gamma)=(\gamma \uparrow \textrm{stab}(e_1^0))+(\gamma \uparrow \textrm{stab}(e_1^1))-(\gamma \uparrow \textrm{stab}(e_1^1))-(\gamma \uparrow \textrm{stab}(e_1^0))=0,$$
$$\Phi_1(\beta_0)=(\beta_0 \uparrow \textrm{stab}(e_0^1))-(\beta_0 \uparrow \textrm{stab}(e_0^0))=\alpha^1_1+\alpha^2_1-\alpha^1_0-\alpha^2_0-\alpha^3_0-\alpha^4_0,$$
$$\Phi_1(\beta_1)=(\beta_1 \uparrow \textrm{stab}(e_0^2))-(\beta_1 \uparrow \textrm{stab}(e_0^1))=\alpha^1_2+\alpha^2_2+\alpha^3_2+\alpha^4_2-\alpha^1_1-\alpha^2_1.$$

As $\Phi_2=0$, it is deduced immediately that $H_2^{\mathcal{F}}(\mathbf{p4},R_{\mathbb{C}})= \Z$. In turn, the SMF of the matrix of $\Phi_1$ gives $(1,1)$ as invariant factors, so $H_1^{\mathcal{F}}(\mathbf{p4},R_{\mathbb{C}})= 0$ and $H_0^{\mathcal{F}}(\mathbf{cmm},R_{\mathbb{C}})= \Z^8$.

It is clear that $[\gamma]$ is a basis of $H_2^{\mathcal{F}}(\mathbf{p4},R_{\mathbb{C}})$, while the shape of the matrix $Q$ associated to the SNF of the matrix of $\Phi_1$ implies that a basis for $H_0^{\mathcal{F}}(\mathbf{p4},R_{\mathbb{C}})$ is given by $([\alpha_0^1],[\alpha_0^2],[\alpha_0^3],[\alpha_0^4],[\alpha_1^2],[\alpha_2^2],[\alpha_2^3],[\alpha_2^4])$.

\subsection{The group $\mathbf{p4m}$.}


All the eight triangles in the picture in Figure \ref{Doris} whose edges are reflection will be elements of the class of equivariant 2-cells. We choose as a representative $e_2$ the only one whose lowest side is the left half of the the lowest horizontal side of the square. Starting from the left vertex of this half and counting clockwise, we denote by $O$, $P$ and $Q$ the vertices of the triangle. Then the representatives $e_0^0$, $e_0^1$ and $e_0^2$ of the classes of 0-cells will be respectively identified with this three points. There are also three classes for $1$-cells, so we make the segments $OP$, $PQ$ and $QO$ correspond to the representatives $e_1^0$, $e_1^1$ and $e_1^2$. We compute the boundaries for these cells:

$$\partial(e_2)=e_1^0+e_1^1+e_1^2,$$
$$\partial(e_1^0)=e_0^1-e_0^0,$$
$$\partial(e_1^1)=e_0^2-e_0^1,$$
$$\partial(e_1^2)=e_0^0-e_0^2.$$

The stabilizers of $e_0^0$ and $e_0^1$ are generated by a 4-rotation and an independent reflection, so they are both isomorphic to $D_4$. In turn, the stabilizer of $e_0^2$ is generated by two reflections, and hence it is isomorphic to $D_2$. As all the 1-cells lie in reflection axes, the corresponding stabilizers are isomorphic to $C_2$. Hence, we obtain the following Bredon complex:

$$0\ra \bigoplus_{i=1}^2\beta_0^i\bigoplus_{i=1}^2\beta_1^i\bigoplus_{i=1}^2\beta_2^i\ra \bigoplus_{i=1}^5\alpha_0^i\bigoplus_{i=1}^5\alpha_1^i\bigoplus_{i=1}^4\alpha_2^i.$$

Taking account of lines 6, 7, 10, 11, 12 and 13 in Table \ref{Stabilizers}, the differentials  of the chain complex are given by:

$$\Phi_2(\gamma)=(\gamma \uparrow \textrm{stab}(e_1^0))+(\gamma \uparrow \textrm{stab}(e_1^1))+(\gamma \uparrow \textrm{stab}(e_1^2))=\beta_0^1+\beta_0^2+\beta_1^1+\beta_1^2+\beta_2^1+\beta_2^2,$$
$$\Phi_1(\beta_0^1)=(\beta_0^1 \uparrow \textrm{stab}(e_0^1))-(\beta_0^1 \uparrow \textrm{stab}(e_0^0))=\alpha^1_1+\alpha^3_1+\alpha^5_1-\alpha^1_0-\alpha^3_0-\alpha^5_0,$$
$$\Phi_1(\beta_0^2)=(\beta_0^2 \uparrow \textrm{stab}(e_0^1))-(\beta_0^2 \uparrow \textrm{stab}(e_0^0))=\alpha^2_1+\alpha^4_1+\alpha^5_1-\alpha^2_0-\alpha^4_0-\alpha^5_0,$$
$$\Phi_1(\beta_1^1)=(\beta_1^1 \uparrow \textrm{stab}(e_0^2))-(\beta_1^1 \uparrow \textrm{stab}(e_0^1))=\alpha^1_2+\alpha^2_2-\alpha^1_1-\alpha^4_1-\alpha^5_1,$$
$$\Phi_1(\beta_1^2)=(\beta_1^2 \uparrow \textrm{stab}(e_0^2))-(\beta_1^2 \uparrow \textrm{stab}(e_0^1))=\alpha^3_2+\alpha^4_2-\alpha^2_1-\alpha^3_1-\alpha^5_1,$$
$$\Phi_1(\beta_2^1)=(\beta_2^1 \uparrow \textrm{stab}(e_0^0))-(\beta_2^1 \uparrow \textrm{stab}(e_0^2))=\alpha^1_0+\alpha^4_0+\alpha^5_0-\alpha^1_2-\alpha^3_2,$$
$$\Phi_1(\beta_2^2)=(\beta_2^2 \uparrow \textrm{stab}(e_0^0))-(\beta_2^2 \uparrow \textrm{stab}(e_0^2))=\alpha^2_0+\alpha^3_0+\alpha^5_0-\alpha^2_2-\alpha^4_2.$$

Here, the SMF of the matrices of $\Phi_2$ and $\Phi_1$ gives the invariant factors $(1)$ and $(1,1,1,1,1)$, respectively. Then,  $H_2^{\mathcal{F}}(\mathbf{p4m},R_{\mathbb{C}})= 0$, $H_1^{\mathcal{F}}(\mathbf{p4m},R_{\mathbb{C}})=0$ and $H_0^{\mathcal{F}}(\mathbf{p4m},R_{\mathbb{C}})=\Z^9$.

Again a basis of $H_0^{\mathcal{F}}(\mathbf{p4m},R_{\mathbb{C}})$ can be extracted of the last column of the auxiliary matrix $Q$. Such a basis is $([\alpha_0^4],[\alpha_0^5],[\alpha_1^3],[\alpha_1^4],[\alpha_1^5],[\alpha_2^1],[\alpha_2^2],[\alpha_2^3],[\alpha_2^4])$.

\subsection{The group $\mathbf{p4g}$.}


Here our representative for the class of equivariant 2-cells will be the triangle whose vertices are the middle point of the left side of the (big) square, the center of that square and the middle point of the lowest side of that square. We will denote these vertices by $O$, $P$ and $Q$ respectively. In particular, we identify $O$ and $P$ with our representatives $e_0^0$ and $e_0^1$ for the two classes of 0-cells, while $Q$ is the image of $O$ under a counterclockwise 4-rotation $t$ whose center is $P$. There are also two classes of equivariant 1-cells, whose representatives $e_1^0$ and $e_1^1$ we respectively identify with $OP$ and $QO$. For the other side of the triangle, we have $QP=t(OP)$. The boundaries are then defined in this way:

$$\partial(e_2)=e_1^0-te_1^0+e_1^1,$$
$$\partial(e_1^0)=e_0^1-e_0^0,$$
$$\partial(e_1^1)=e_0^0-te_0^0,$$

As $e_0^1$ is a center of 4-rotation, its stabilizer is $C_4$. In turn, the stabilizer of $e_0^0$ is generated by a 2-rotation and a reflection, so it is isomorphic to $D_2$. On the other hand, the unique 1-cell with non-trivial isotropy is $e_1^1$, which lies in a reflection axis and then has $C_2$ as stabilizer. Let us write now the Bredon chain complex:

$$0\ra \Z\gamma \ra \Z\beta_0\oplus\Z\beta_1^1\oplus\beta_1^2\ra \bigoplus_{i=1}^4\alpha_0^i\bigoplus_{i=1}^4\alpha_1^i\ra 0.$$

Now we can can compute the differentials, taking account of lines 1, 6 and 7 in Table \ref{Stabilizers} :

$$\Phi_2(\gamma)=(\gamma \uparrow \textrm{stab}(e_1^0))-(\gamma \uparrow \textrm{stab}(e_1^0))+(\gamma \uparrow \textrm{stab}(e_1^1))=\beta_1^1+\beta_1^2,$$
$$\Phi_1(\beta_0)=(\beta_0 \uparrow \textrm{stab}(e_0^1))-(\beta_0 \uparrow \textrm{stab}(e_0^0))=\alpha^1_1+\alpha^2_1+\alpha^3_1+\alpha^4_1-\alpha^1_0-\alpha^2_0-\alpha^3_0-\alpha^4_0,$$
$$\Phi_1(\beta_1^1)=(\beta_1^1 \uparrow \textrm{stab}(e_0^0))-(\beta_1^1 \uparrow \textrm{stab}(e_0^0))=\alpha^1_0+\alpha^2_0-\alpha^1_0-\alpha^3_0=\alpha^2_0-\alpha^3_0,$$
$$\Phi_1(\beta_1^2)=(\beta_1^2 \uparrow \textrm{stab}(e_0^0))-(\beta_1^2 \uparrow \textrm{stab}(e_0^0))=\alpha^3_0+\alpha^4_0-\alpha^2_0-\alpha^4_0=\alpha^3_0-\alpha^2_0.$$

Observe that the two induced characters denoted by $(\beta_1^1 \uparrow \textrm{stab}(e_0^0))$ in the expression of $\Phi_1(\beta_1^1)$ are not the same, because they are induced from different inclusions $\Z/2\hookrightarrow D_4$. Same happens with $\beta_1^2$.

After computing the SMF of the matrices of $\Phi_2$ and $\Phi_1$, we obtain the invariant factors $(1)$ and $(1,1)$, respectively. Hence, we have  $H_2^{\mathcal{F}}(\mathbf{p4g},R_{\mathbb{C}})= 0$, $H_1^{\mathcal{F}}(\mathbf{p4g},R_{\mathbb{C}})=0$ and $H_0^{\mathcal{F}}(\mathbf{p4g},R_{\mathbb{C}})=\Z^6$.

From the matrix $Q$ obtained in the computation of the SMF of the matrix of $\Phi_1$ we conclude that a basis for $H_0^{\mathcal{F}}(\mathbf{p4g},R_{\mathbb{C}})$ is given by $([\alpha_0^1], [\alpha_0^2], [\alpha_0^4], [\alpha_1^1], [\alpha_1^2], [\alpha_1^4])$.

\subsection{The group $\mathbf{p3}$.}


For the group $\mathbf{p3}$ a representative $e_2$ for the class of equivariant 2-cells is the rhombus of the picture. We order the vertices clockwise starting from the upper left, and denote them as usual by $P$, $Q$, $R$ and $S$. Then representatives $e_0^0$, $e_0^1$ and $e_0^2$ of the three classes of 0-cells are given by $P$, $Q$ and $R$, and a 3-rotation $t$ around $O$ takes $P$ to $S$. In turn, there are two classes $e_1^0$ and $e_1^1$ of 1-cells, which can be respectively identified with $OP$ and $PQ$. Observe that $t(OP)=t(OS)$ and $t(PQ)=t(SQ)$. We have the boundaries:

$$\partial(e_2)=e_1^0+e_1^1-te_1^1-te_1^0,$$
$$\partial(e_1^0)=e_0^1-e_0^0,$$
$$\partial(e_1^1)=e_0^2-e_0^1.$$

The three vertices are 3-rotation centers, and their stabilizers are isomorphic to $C_3$. On the other hand, the group acts freely in each class of equivariant $1$-cells, and hence the Bredon complex has the form:

$$0\ra \Z\gamma \ra \Z\beta_0\oplus\Z\beta_1 \ra \bigoplus_{i=1}^3\Z\alpha_0^i\bigoplus_{i=1}^3\Z\alpha_1^i\bigoplus_{i=1}^3\Z\alpha_2^i\ra 0.$$

Computing the differentials of the complex, and taking into account line 1 in Table \ref{Stabilizers} we obtain:

$$\Phi_2(\gamma)=(\gamma \uparrow \textrm{stab}(e_1^0))+(\gamma \uparrow \textrm{stab}(e_1^1))-(\gamma \uparrow \textrm{stab}(e_1^1))-(\gamma \uparrow \textrm{stab}(e_1^0))=0,$$
$$\Phi_1(\beta_0)=(\beta_0 \uparrow \textrm{stab}(e_0^1))-(\beta_0 \uparrow \textrm{stab}(e_0^0))=\alpha^1_1+\alpha^2_1-\alpha^1_1-\alpha^3_0=\alpha^2_0-\alpha^3_0,$$
$$\Phi_1(\beta_1)=(\beta_1 \uparrow \textrm{stab}(e_0^2))-(\beta_1 \uparrow \textrm{stab}(e_0^1))=\alpha^1_2+\alpha^2_2+\alpha^3_2-\alpha^1_1-\alpha^2_1-\alpha^3_1.$$

As $\Phi_2$ is trivial, we have $H_2^{\mathcal{F}}(\mathbf{p3},R_{\mathbb{C}})=\Z$. On the other hand, the SNF of the matrix of $\Phi_1$ produces the invariant factors $(1,1)$, whence $H_1^{\mathcal{F}}(\mathbf{p3},R_{\mathbb{C}})=0$ and $H_0^{\mathcal{F}}(\mathbf{p3},R_{\mathbb{C}})=\Z^7.$ From the auxiliary matrix of the SNF we obtain a basis $([\alpha_0^2],[\alpha_0^3],[\alpha_1^1],[\alpha_1^2],[\alpha_1^3],[\alpha_2^1],[\alpha_2^2])$ for $H_0^{\mathcal{F}}(\mathbf{p3},R_{\mathbb{C}})$.

\subsection{The group $\mathbf{p3m1}$.}


A representative $e_2$ for the equivalence class of equivariant 2-cells for this action is the equilateral triangle whose vertices are the upper left vertex in the picture of Figure \ref{Doris} and the two closest rotation centers. We name this three vertices as $O$, $P$ and $Q$, starting from the upper left and counting clockwise. The representatives $e_0^0$, $e_0^1$ and $e_0^2$ of the three classes of 1-cells are identified with these three vertices in that order. There are also three classes of 1-cells, $e_1^0$, $e_1^1$ and $e_1^2$, which we identify respectively with $OP$, $PQ$ and $QO$. Now we can describe the boundaries between the cells:

$$\partial(e_2)=e_1^0+e_1^1+e_1^2,$$
$$\partial(e_1^0)=e_0^1-e_0^0,$$
$$\partial(e_1^1)=e_0^2-e_0^1,$$
$$\partial(e_1^2)=e_0^0-e_0^2.$$

As in the previous group, the vertices are centers of 3-rotations, and moreover the sides of the triangle lie on rotation axes. Then, the stabilizers of the 0-cells are isomorphic to $D_3$, and the stabilizers of the 1-cells are isomorphic to $C_2$. So we have the complex:

$$0\ra \Z\gamma \ra \bigoplus_{i=1}^2\beta_0^i\bigoplus_{i=1}^2\beta_1^i\bigoplus_{i=1}^2\beta_2^i\ra \bigoplus_{i=1}^4\alpha_0^i\bigoplus_{i=1}^4\alpha_1^i\bigoplus_{i=1}^4\alpha_2^i\ra 0.$$

Now we can get the differentials of the chain complex, recall lines 1, 8 and 9 of Table \ref{Stabilizers}:

$$\Phi_2(\gamma)=(\gamma \uparrow \textrm{stab}(e_1^0))+(\gamma \uparrow \textrm{stab}(e_1^1))+(\gamma \uparrow \textrm{stab}(e_1^2))=\beta_0^1+\beta_0^2+\beta_1^1+\beta_1^2+\beta_2^1+\beta_2^2.$$
$$\Phi_1(\beta_0^1)=(\beta_0^1 \uparrow \textrm{stab}(e_0^1))-(\beta_0^1 \uparrow \textrm{stab}(e_0^0))=\alpha^1_1+\alpha^3_1-\alpha^1_0-\alpha^3_0,$$
$$\Phi_1(\beta_0^2)=(\beta_0^2 \uparrow \textrm{stab}(e_0^1))-(\beta_0^2 \uparrow \textrm{stab}(e_0^0))=\alpha^2_1+\alpha^3_1-\alpha^2_0-\alpha^3_0,$$
$$\Phi_1(\beta_1^1)=(\beta_1^1 \uparrow \textrm{stab}(e_0^2))-(\beta_1^1 \uparrow \textrm{stab}(e_0^1))=\alpha^1_2+\alpha^3_2-\alpha^1_1-\alpha^3_1,$$
$$\Phi_1(\beta_1^2)=(\beta_1^2 \uparrow \textrm{stab}(e_0^2))-(\beta_1^2 \uparrow \textrm{stab}(e_0^1))=\alpha^2_2+\alpha^3_2-\alpha^2_1-\alpha^3_1,$$
$$\Phi_1(\beta_2^1)=(\beta_2^1 \uparrow \textrm{stab}(e_0^0))-(\beta_2^1 \uparrow \textrm{stab}(e_0^2))=\alpha^1_0+\alpha^3_0-\alpha^1_2-\alpha^3_2,$$
$$\Phi_1(\beta_2^2)=(\beta_2^2 \uparrow \textrm{stab}(e_0^0))-(\beta_2^2 \uparrow \textrm{stab}(e_0^2))=\alpha^2_0+\alpha^3_0-\alpha^2_2-\alpha^3_2.$$

The respective SNF of the matrices of $\Phi_2$ and $\Phi_1$ respectively produce invariant factors $(1)$ and $(1,1,1,1)$. Then the Bredon homology groups are $H_2^{\mathcal{F}}(\mathbf{p3m1},R_{\mathbb{C}})=0$, $H_1^{\mathcal{F}}(\mathbf{p3m1},R_{\mathbb{C}})=\Z$ and $H_0^{\mathcal{F}}(\mathbf{p3m1},R_{\mathbb{C}})=\Z^5$.

The computation of the auxiliary matrices of the SNF in particular show that a basis for $H_1^{\mathcal{F}}(\mathbf{p3m1},R_{\mathbb{C}})$ is given by $[\beta_0^1+\beta_1^1+\beta_2^1]$ and a basis for $H_0^{\mathcal{F}}(\mathbf{p3m1},R_{\mathbb{C}})$ is $([\alpha_0^3],[\alpha_1^1],[\alpha_1^2],[\alpha_1^3],[\alpha_2^3])$.

\subsection{The group $\mathbf{p31m}$.}


In the picture we see a rhombus divided in two equilateral triangles. As the representative $e_2$ of the class of $2$-cells it can be taken the triangle T whose vertices are the center and the two lowest vertices of the left triangle of the rhombus. We denote by $O$, $P$ and $Q$ the vertices of T, starting in the one in the left-hand side and counting clockwise. Now, the two classes of 0-cells for the action will have representatives $e_0^0$ and $e_0^1$ identified with $P$ and $Q$, and we take account of the fact that $Q=t(O)$, being $t$ the 3-rotation with center $P$. In turn, there are two classes of $1$-cells, represented by $e_1^0$ and $e_1^1$, which we identify with $OP$ and $QO$. Observe that $t(OP)=QP$. Now we can describe the boundaries here:

$$\partial(e_2)=e_1^0+te_1^0+e_1^1,$$
$$\partial(e_1^0)=e_0^1-e_0^0,$$
$$\partial(e_1^1)=e_0^0-te_0^0.$$

The stabilizer of $e_0^0$ is generated by a 3-rotation and reflection, and hence is isomorphic to $D_3$. In turn, the stabilizer of $e_0^2$ is generated by a 3-rotation and is then isomorphic to $C_3$. On the other hand, the group acts freely over the class of $e_1^0$, while the stabilizer of $e_1^1$ is $C_2$, as this representative lies in a reflection axis. Then the Bredon chain complex is:

$$0\ra \Z\gamma \ra \Z\beta_0\oplus\Z\beta_1^1\oplus\Z\beta_1^2 \ra\bigoplus_{i=1}^3 \Z\alpha_0^i \bigoplus_{i=1}^3 \Z\alpha_1^i\ra 0.$$

Now let us compute the differentials of the complex, taking account of lines 1, 3, 8 and 9 of Table \ref{Stabilizers}:

$$\Phi_2(\gamma)=(\gamma \uparrow \textrm{stab}(e_1^0))-(\gamma \uparrow \textrm{stab}(e_1^0))+(\gamma \uparrow \textrm{stab}(e_1^1))=\beta_1^1+\beta_1^2,$$
$$\Phi_1(\beta_0)=(\beta_0 \uparrow \textrm{stab}(e_0^1))-(\beta_0 \uparrow \textrm{stab}(e_0^0))=\alpha^1_1+\alpha^2_1+\alpha^3_1-\alpha^1_0-\alpha^2_0-\alpha^3_0,$$
$$\Phi_1(\beta_1^1)=(\beta_1^1 \uparrow \textrm{stab}(e_0^0))-(\beta_1^1 \uparrow \textrm{stab}(e_0^0))=0,$$
$$\Phi_1(\beta_1^2)=(\beta_1^2 \uparrow \textrm{stab}(e_0^0))-(\beta_1^2 \uparrow \textrm{stab}(e_0^0))=0.$$

Observe that, unlike what happens in the case of $\mathbf{p4g}$, here the differences $(\beta_1^i \uparrow \textrm{stab}(e_0^0))-(\beta_1^i \uparrow \textrm{stab}(e_0^0)$ are trivial because in $D_3$ there is only one conjugation class of elements of order 2.

Now the computation of the SNF for the two differentials produce the invariant factors $(1)$ and $(1)$ for $\Phi_2$ and $\Phi_1$, respectively. Then, $H_2^{\mathcal{F}}(\mathbf{p31m},R_{\mathbb{C}})=0$, $H_1^{\mathcal{F}}(\mathbf{p31m},R_{\mathbb{C}})=\Z$ and $H_0^{\mathcal{F}}(\mathbf{p31m},R_{\mathbb{C}})=\Z^5$.

The auxiliary matrices of the SNF produce the bases $[\beta_1^1]$ for $H_1^{\mathcal{F}}(\mathbf{p31m},R_{\mathbb{C}})$ and $([\alpha_0^2],[\alpha_0^3],[\alpha_1^1],[\alpha_1^2],[\alpha_1^3])$ for $H_0^{\mathcal{F}}(\mathbf{p31m},R_{\mathbb{C}})$.

\subsection{The group $\mathbf{p6}$.}


The representative $e_2$ the class of equivariant 2-cells by the action of $\mathbf{p6}$ will be exactly the same triangle as in the previous group. We keep the names of the vertices $O$, $P$ and $Q$, and we consider another vertex $R$, the middle point of the lowest side of the triangle. There will be three of classes of equivalence of 0-cells, whose representatives $e_0^0$, $e_0^1$ and $e_0^2$ are identified with $O$, $P$ and $R$. The vertex $Q$ is now the image of $O$ under the 2-rotation $r_1$ centered in $R$ (and also under the 3-rotation $r_2$ centered in $P$). There are also two classes of 1-cells, whose representatives $e_1^0$ and $e_1^1$ are identified respectively with the segments $OP$ and $QR$. Observe that $r_2(OP)=QP$ and $r_1(QR)=OR$. We can now write the boundaries for the cells:

$$\partial(e_2)=e_1^0+r_2e_1^0+e_1^1+r_1e_1^1,$$
$$\partial(e_1^0)=e_0^1-e_0^0,$$
$$\partial(e_1^1)=e_0^2-te_0^0.$$

 Observe that $e_0^0$, $e_0^1$ and $e_0^2$ are respectively $6$-, $3$- and $2$- rotation centers, and their stabilizers are respectively isomorphic to $C_6$, $C_3$ and $C_2$. On the other hand, the group $\mathbf{p6}$ act freely over each of the classes of $1$-cells, so we have the following Bredon chain complex:

$$0\ra \Z\gamma \ra \Z\beta_0\oplus\Z\beta_1 \rightarrow \bigoplus_{i=1}^6 \alpha_0^i\bigoplus_{i=1}^3 \alpha_1^i\bigoplus_{i=1}^2 \alpha_2^i\ra 0.$$

Now we compute the differentials, taking into account line 1 of Table \ref{Stabilizers}:

$$\Phi_2(\gamma)=(\gamma \uparrow \textrm{stab}(e_1^0))-(\gamma \uparrow \textrm{stab}(e_1^0))+(\gamma \uparrow \textrm{stab}(e_1^1))-(\gamma \uparrow \textrm{stab}(e_1^1))=0,$$
$$\Phi_1(\beta_0)=(\beta_0 \uparrow \textrm{stab}(e_0^1))-(\beta_0 \uparrow \textrm{stab}(e_0^0))=\alpha^1_1+\alpha^2_1+\alpha^3_1-\alpha^1_0-\alpha^2_0-\alpha^3_0-\alpha^4_0-\alpha^5_0-\alpha^6_0,$$
$$\Phi_1(\beta_1)=(\beta_1 \uparrow \textrm{stab}(e_0^2))-(\beta_1 \uparrow \textrm{stab}(e_0^0))=\alpha^1_2+\alpha^2_2-\alpha^1_0-\alpha^2_0-\alpha^3_0-\alpha^4_0-\alpha^5_0-\alpha^6_0.$$

As the first homomorphism is trivial, we get $H_2^{\mathcal{F}}(\mathbf{p6},R_{\mathbb{C}})=\Z$. On the other hand, the SNF of the matrix of $\Phi_1$ has $(1,1)$ as invariant factors, and thus $H_1^{\mathcal{F}}(\mathbf{p6},R_{\mathbb{C}})=0$ and $H_0^{\mathcal{F}}(\mathbf{p6},R_{\mathbb{C}})=\Z^9$.

Clearly a basis for $H_2^{\mathcal{F}}(\mathbf{p6},R_{\mathbb{C}})$ is given by $[\gamma]$, while the auxiliary computations of the SNF provide the basis of $H_1^{\mathcal{F}}(\mathbf{p6},R_{\mathbb{C}})$ given by $([\alpha_0^2],[\alpha_0^3],[\alpha_0^4],[\alpha_0^5],[\alpha_0^6],[\alpha_1^2],[\alpha_1^3],[\alpha_2^1],[\alpha_2^2])$.

\subsection{The group $\mathbf{p6m}$.}


To get our representative $e_2$ for the class of equivariant 2-cells in this group, we take in the previous group $\mathbf{p6}$ the triangle defined by the vertices $O$, $P$ and $R$ there. We keep the name of the vertices here, and identify them in this order with the representatives $e_0^0$, $e_0^1$ and $e_0^2$ of the three equivalence classes of $0$-cells. We also identify the representatives $e_1^0$, $e_1^1$ and $e_1^2$ of the three classes of 1-cells with the segments $OP$, $PR$ and $RO$. Now the boundaries are:

$$\partial(e_2)=e_1^0+e_1^1+e_1^2,$$
$$\partial(e_1^0)=e_0^1-e_0^0,$$
$$\partial(e_1^1)=e_0^2-e_0^1,$$
$$\partial(e_1^2)=e_0^0-e_0^2.$$

Again $e_0^0$, $e_0^1$ and $e_0^2$ are respectively 6-, 3- and 2-rotation centers. In addition, all of them lie in a reflection axis. Hence, the stabilizers of these vertices are respectively $D_6$, $D_3$ and $D_2$. As the representatives of the 1-cells all lie in some reflection axis, their stabilizers are all isomorphic to $C_2$. So we have the following Bredon chain complex:

$$0\ra \Z\gamma \ra \bigoplus_{i=1}^2\beta_0^i\bigoplus_{i=1}^2\beta_1^i\bigoplus_{i=1}^2\beta_2^i\ra \bigoplus_{i=1}^6\alpha_0^i\bigoplus_{i=1}^3\alpha_1^i\bigoplus_{i=1}^4\alpha_2^i\ra 0.$$

We compute the differentials for this complex, taking into account lines 8, 9, 14, 15, 16 and 17 of Table \ref{Stabilizers}:

$$\Phi_2(\gamma)=(\gamma \uparrow \textrm{stab}(e_1^0))+(\gamma \uparrow \textrm{stab}(e_1^1))+(\gamma \uparrow \textrm{stab}(e_1^2)=\beta_0^1+\beta_0^2+\beta_1^1+\beta_1^2+\beta_2^1+\beta_2^2,$$
$$\Phi_1(\beta_0^1)=(\beta_0^1 \uparrow \textrm{stab}(e_0^1))-(\beta_0^1 \uparrow \textrm{stab}(e_0^0))=\alpha^1_1+\alpha^3_1-\alpha^1_0-\alpha^3_0-\alpha^5_0-\alpha^6_0,$$
$$\Phi_1(\beta_0^2)=(\beta_0^2 \uparrow \textrm{stab}(e_0^1))-(\beta_0^2 \uparrow \textrm{stab}(e_0^0))=\alpha^2_1+\alpha^3_1-\alpha^2_0-\alpha^4_0-\alpha^5_0-\alpha^6_0,$$
$$\Phi_1(\beta_1^1)=(\beta_1^1 \uparrow \textrm{stab}(e_0^2))-(\beta_1^1 \uparrow \textrm{stab}(e_0^1))=\alpha^1_2+\alpha^2_2-\alpha^1_1-\alpha^3_1,$$
$$\Phi_1(\beta_1^2)=(\beta_1^2 \uparrow \textrm{stab}(e_0^2))-(\beta_1^2 \uparrow \textrm{stab}(e_0^1))=\alpha^3_2+\alpha^4_2-\alpha^2_1-\alpha^3_1,$$
$$\Phi_1(\beta_2^1)=(\beta_2^1 \uparrow \textrm{stab}(e_0^0))-(\beta_2^1 \uparrow \textrm{stab}(e_0^2))=\alpha^1_0+\alpha^4_0+\alpha^5_0+\alpha^6_0-\alpha^1_2-\alpha^3_2,$$
$$\Phi_1(\beta_2^2)=(\beta_2^2 \uparrow \textrm{stab}(e_0^0))-(\beta_2^2 \uparrow \textrm{stab}(e_0^2))=\alpha^2_0+\alpha^3_0+\alpha^5_0+\alpha^6_0-\alpha^2_2-\alpha^4_2.$$

The calculation of the SNF for $\Phi_2$ and $\Phi_1$ produce the invariant factors $(1)$ and $(1,1,1,1,1)$, respectively. Then we have $H_2^{\mathcal{F}}(\mathbf{p6m},R_{\mathbb{C}})=0$, $H_1^{\mathcal{F}}(\mathbf{p6m},R_{\mathbb{C}})=0$ and $H_0^{\mathcal{F}}(\mathbf{p6m},R_{\mathbb{C}})=\Z^8$.

The auxiliary matrix $Q$ of the SNF permits to identify a basis of $H_0^{\mathcal{F}}(\mathbf{p6m},R_{\mathbb{C}})$ $[\beta_1^1]$, which is given by $([\alpha_0^4],[\alpha_0^5],[\alpha_0^6],[\alpha_1^1],[\alpha_1^3],[\alpha_2^1],[\alpha_2^3],[\alpha_2^4])$.

\begin{remark}

Observe that applying \cite[Theorem 5.27]{MiVa03} to the Bredon homology groups of the wallpaper groups and recalling that the Baum-Connes conjecture holds for these groups, we recover the $K$-theory computations of (\cite{LuSt00}, Section 5). Hence, our results can also be interpreted as a way to approach directly the left-hand side of the conjecture in the case of wallpaper groups.

\end{remark}

\begin{table}[t]
\begin{center}
\begin{tabular}{| c | c | c | c | c | c |}
\hline
Group & $H_2$ & $H_1$ & Basis $H_1$ & $H_0$ & Basis $H_0$ \\ \hline

$\mathbf{p1} $  & $\Z$ & $\Z^2$ & $([\beta_1],[\beta_2])$   & $\Z$ & $[\alpha]$  \\ \hline
$\mathbf{p2} $  & $ \Z $ & $ 0 $ & - & $ \Z^5 $ & $ ([\alpha_0^1], [\alpha_0^2], [\alpha_1^2], [\alpha_2^2], [\alpha_3^2]) $\\ \hline
$\mathbf{pm} $  & $ 0 $ & $ \Z^3 $ & $ ([\beta_1^1], [\beta_1^2], [\beta_2^1]) $ & $ \Z^3 $ & $ ([\alpha_0^2], [\alpha_1^1], [\alpha_1^2]) $\\ \hline
$\mathbf{pg} $  & $ 0 $ & $ \Z/2\oplus \Z $ & $ ([\beta_0], [\beta_1]) $ & $ \Z $ & $ [\alpha] $\\ \hline
$\mathbf{cm} $  & $ 0 $ & $ \Z^2$ & $ ([\beta_1^1],[\beta_1^2]) $ & $ \Z^2 $ & $ ([\alpha^1],[\alpha^2]) $\\ \hline
$\mathbf{pmm} $  & $ 0 $ & $ 0 $ & - & $ \Z^9 $ & $ ([\alpha_0^3], [\alpha_1^3], [\alpha_1^4], [\alpha_2^3], [\alpha_2^4], [\alpha_3^1], [\alpha_3^2], [\alpha_3^3], [\alpha_3^4]) $\\ \hline
$\mathbf{pmg} $  & $ 0 $ & $ \Z $ & $ [\beta_1^1+\beta_2^1] $ & $ \Z^4 $ & $ ([\alpha_1^2], [\alpha_2^1], [\alpha_2^2], [\alpha_3^2]) $\\ \hline
$\mathbf{pgg} $  & $ 0 $ & $ \Z /2 $ & $ [\beta_0] $ & $ \Z^3 $ & $ ([\alpha_1^2], [\alpha_2^1], [\alpha_2^2]) $\\ \hline
$\mathbf{cmm} $  & $ 0 $ & $ 0 $ & - & $ \Z^6 $ & $ ([\alpha_0^1+\alpha_0^2], [\alpha_0^3], [\alpha_0^4], [\alpha_1^1], [\alpha_1^3], [\alpha_2^2]) $\\ \hline
$\mathbf{p4} $  & $ \Z $ & $ 0 $ & - & $ \Z^8 $ & $ ([\alpha_0^1],[\alpha_0^2],[\alpha_0^3],[\alpha_0^4],[\alpha_1^2],[\alpha_2^2],[\alpha_2^3],[\alpha_2^4]) $\\ \hline
$\mathbf{p4m} $ & $ 0 $ & $ 0 $ & - & $ \Z^9 $ & $ ([\alpha_0^4],[\alpha_0^5],[\alpha_1^3],[\alpha_1^4],[\alpha_1^5],[\alpha_2^1],[\alpha_2^2],[\alpha_2^3],[\alpha_2^4]) $\\ \hline
$\mathbf{p4g} $ & $ 0 $ & $ 0 $ & - & $ \Z^6 $ & $ ([\alpha_0^1], [\alpha_0^2], [\alpha_0^4], [\alpha_1^1], [\alpha_1^2], [\alpha_1^4]) $\\ \hline
$\mathbf{p3} $ & $ \Z $ & $ 0 $ & - & $ \Z^7 $ & $ ([\alpha_0^2],[\alpha_0^3],[\alpha_1^1],[\alpha_1^2],[\alpha_1^3],[\alpha_2^1],[\alpha_2^2]) $\\ \hline
$\mathbf{p3m1} $ & $ 0 $ & $ \Z $ & $ [\beta_0^1+\beta_1^1+\beta_2^1] $ & $ \Z^5 $ & $ ([\alpha_0^3],[\alpha_1^1],[\alpha_1^2],[\alpha_1^3],[\alpha_2^3]) $\\ \hline
$\mathbf{p31m} $ & $ 0 $ & $ \Z $ & $ [\beta_1^1] $ & $ \Z^5 $ & $ ([\alpha_0^2],[\alpha_0^3],[\alpha_1^1],[\alpha_1^2],[\alpha_1^3]) $\\ \hline
$\mathbf{p6} $ & $ \Z $ & $ 0 $ & - & $ \Z^9 $ & $ ([\alpha_0^2],[\alpha_0^3],[\alpha_0^4],[\alpha_0^5],[\alpha_0^6],[\alpha_1^2],[\alpha_1^3],[\alpha_2^1],[\alpha_2^2]) $\\ \hline
$\mathbf{p6m} $ & $ 0 $ & $ 0 $ & - & $ \Z^8 $ & $ ([\alpha_0^4],[\alpha_0^5],[\alpha_0^6],[\alpha_1^1],[\alpha_1^3],[\alpha_2^1],[\alpha_2^3],[\alpha_2^4]) $\\ \hline

\end{tabular}
\caption{Bredon homology}
\label{Bredon homology}
\end{center}
\end{table}

\end{document}